\documentclass[11pt]{article}

\usepackage{graphicx,latexsym,amsmath,amsfonts,amssymb,dsfont}
\usepackage{amsthm}
\usepackage{enumerate}
\usepackage{epsfig}
\usepackage{verbatim}
\usepackage{xcolor}
\usepackage{enumitem}
\usepackage{mathrsfs}
\usepackage{hyperref,cleveref}
\usepackage{cite}
\usepackage[title,titletoc,toc]{appendix}

\def\td{\mathrm{d}} \def\tD{\mathrm{D}}

\def\tr{\mathrm{tr}}
\def\det{\mathrm{det}}
\def\ex{\mathds{E}}
\def\eps{\varepsilon}
\def\prb{\mathds{P}} \def\qrb{\mathds{Q}}
\def\one{\mathds{1}}
\def\wass{\mathds{W}}
\def\bx{\bar{X}}
\def\NN{\mathbb{N}}
\def\ZZ{\mathbb{Z}}

\def\RR{\mathbb{R}}
\def\CC{\mathbb{C}}
\def\lb{\left(} \def\rb{\right)}
\def\lv{\left|} \def\rv{\right|}
\def\lvv{\left\|} \def\rvv{\right\|}

\def\lq{\leqslant} \def\gq{\geqslant}

\def\levy{L\'{e}vy} \def\ito{It\={o}}
\def\var{\textrm{var}} \def\poi{\textrm{Poi}}
\def\P{\mathcal{P}} \def\Q{\mathcal{Q}}
\def\F{\mathcal{F}} \def\N{\mathcal{N}}
\def\tilde{\widetilde}
\newcommand\num{\stepcounter{equation}\tag{\theequation}}

\newtheorem*{theorem*}{Theorem}
\newtheorem{theorem}{Theorem}[section]
\newtheorem{lemma}[theorem]{Lemma}
\newtheorem{corollary}[theorem]{Corollary}

\newtheorem{proposition}[theorem]{Proposition}
\newtheorem{remark}[theorem]{Remark}

\newtheorem{innercustomgeneric}{\customgenericname}
\providecommand{\customgenericname}{}
\newcommand{\newcustomtheorem}[2]{%
  \newenvironment{#1}[1]
  {%
   \renewcommand\customgenericname{#2}%
   \renewcommand\theinnercustomgeneric{##1}%
   \innercustomgeneric
  }
  {\endinnercustomgeneric}
}
\newcustomtheorem{assumption}{Assumption}

\usepackage[top=3.0cm, bottom=3.0cm, left=3.0cm, right=3.0cm]{geometry}
\makeatletter \@addtoreset{equation}{section}
\renewcommand{\thesection}{\arabic{section}}
\renewcommand{\theequation}{\thesection.\arabic{equation}}

\title{\textsc{A Multi-Dimensional Central Limit Bound and its Application to the Euler Approximation of {\levy}-SDEs}}

\author{X\={\i}l\'{\i}ng Zh\={a}ng
        \thanks{School of Mathematics, The University of Edinburgh. e-mail: xiling.zhang@ed.ac.uk}}
\date{16 June, 2017}

\begin{document}
\maketitle
\begin{abstract}
In the one-dimensional case Rio \cite{rio2009uppboumindiscenlimthe} gave a concise bound for the central limit theorem in the Vaserstein distances, which is a ratio between some higher moments and some powers of the variance. As a corollary, it gives an estimate for the normal approximation of the small jumps of {\levy} processes, and Fournier \cite{fournier2011SimappLevstodifequ} applied that to the Euler approximation of stochastic differential equations driven by the L\'evy noise. It will be shown in this article that following Davie's idea in \cite{davie2015Polpernordis}, one can generalise Rio's result to the multidimensional case, and have higher-order approximation via the perturbed normal distributions, if Cram\'er's condition and a slightly stronger moment condition are assumed. Fournier's result can then be partially recovered.
\end{abstract}

\section{Introduction}

Given\footnote{Throughout this article $\ZZ^+$ and $\NN$ denote the sets of positive and non-negative integers, respectively.} $d,q,q_1\in\ZZ^+$, let $a\in\RR^q,~B\in\RR^{q\times q_1}$ and $(\Omega,\mathscr{F},\prb)$ be a complete probability space equipped with a filtration $\{\mathscr{F}_t\}_{t\gq0}$ generated by a $q_1$-dimensional Wiener process $\{W_t\}$ and an independent Poisson random measure $N(\td z,\td s)$ on $\RR^q\setminus\{0\}\times[0,\infty)$ with intensity $\nu(\td z)\td s$. Consider the $q$-dimensional {\levy} process on a bounded interval $[0,T]$:
\begin{equation}\label{noise}
Z_t=at+BW_t+\int_0^t\int_{\RR^q\setminus\{0\}}z\tilde{N}(\td z,\td s),
\end{equation}
where $\tilde{N}(\td z,\td s)$ is the compensated Poisson measure. Assume the second moment of the {\levy} measure $\int_{\RR^q\setminus\{0\}}|z|^2\nu(\td z)<\infty$, where $|\cdot|$ denotes the modulus. For $x_0\in\RR^q$ and a bounded Lipschitz function $\sigma:\RR^d\to\RR^{d\times q}$, consider the $d$-dimensional SDE driven by the L\'evy process above:
\begin{equation}\label{sde}
x_t=x_0+\int_0^t\sigma(x_{s-})\td Z_s.
\end{equation}
For $h\in(0,1)$ and $t_k=kh,~k=1,\cdots,[T/h]$, where $[\cdot]$ denotes the integer part, it is known that the standard Euler's approximation,
\begin{equation*}
X_{k+1}:=X_k+\sigma(X_k)\lb Z_{t_{k+1}}-Z_{t_k}\rb, X_0=x_0,
\end{equation*}
converges with rate $1/2$ to the solution of \eqref{sde} in mean-square as $h\to0$ - see e.g. \cite{kohatsuhiga1994EulSchSDEDribySem}, \cite{jacod1998AsyErrDisEulMetStoDifEqu} and \cite{jacod2004EulSchLevDriStoDifEquLimThe}. Although the increments $Z_{t_k}-Z_{t_{k-1}}$ are hard to generate, one may simply ignore the small jumps
\begin{equation}\label{small_jumps}
Z_t^\epsilon:=\int_0^t\int_{0<|z|\lq\epsilon}z\tilde{N}(\td z,\td s),
\end{equation}
for some $\epsilon\in(0,1)$, and show that the mean-square convergence rate is preserved. However, that is not a very economical way of simulation, as pointed out by Fournier \cite{fournier2011SimappLevstodifequ}. Indeed, when the small jumps are completely ignored, the expected computational cost
\begin{equation*}
E_\nu(h)=O\lb h^{-1}+\nu\lb\{|z|>\epsilon\rb\}\rb,
\end{equation*}
can be considerably large. This happens, e.g., when the {\levy} measure $\nu$ behaves like $\alpha$-stable near $0$, i.e. there exist $\tau>0$ and $\alpha\in(0,2)$ s.t. the following condition holds:

\begin{assumption}{$\mathbf{H}(\tau,\alpha)$}\label{alpha-stable}
$\nu(\td z)\simeq|z|^{-q-\alpha}\td z,~\forall0<|z|\lq\tau$.
\end{assumption}

The symbol $\simeq$ is used where both sides are bounded by each other up to a constant factor depending only on $q$. Given condition \ref{alpha-stable}, the set of big jumps has measure $\nu(\{|z|>\epsilon\})\simeq\epsilon^{-\alpha}$, and one has to choose $\epsilon=h^{1/(2-\alpha)}$ to ensure the order $1/2$ of mean-square convergence. As a result $E_\nu(h)=O\lb h^{-1}+h^{\alpha/(\alpha-2)}\rb$ explodes when $\alpha$ is close to $2$.

As a remedy, one may consider approximating the small jumps \eqref{small_jumps} with a normal random variable using the central limit theorem, on which some classical theorems can be found in several books such as \cite{petrov1975SumIndRanVar} and \cite{bhattacharya1976NorAppAsyExp}. Asmussen and Rosi\'nski \cite{asmussen2001AppSmaJumLevProVietowSim} adopted this idea and derived some Berry-Esseen bounds for the normal approximation of the small jumps $Z_1^\epsilon$; they also gave conditions for the weak convergence in the Skorohod space. But their method only works for $q=1$, and the Berry-Esseen-type bounds are not very useful for the strong $L^p$-approximation of {\levy}-SDEs as they only concern the uniform distance between the c.d.f's. Aiming at the Euler approximation of \eqref{sde}, Fournier \cite{fournier2011SimappLevstodifequ} proved that by adding this normal random variable to the Euler scheme the expected computational cost can be controlled (no explosion of $E_\alpha(h)$ near $\alpha=2$), while the $1/2$ convergence rate is still preserved. However, as pointed out himself, the method is also restricted to the case $q=1$.

Such a restriction of dimension only emerged at a key step in \cite{fournier2011SimappLevstodifequ} (Corollary 4.2), borrowed from a result by Rio \cite{rio2009uppboumindiscenlimthe} (Corollary 4.2) on the central limit theorem. The latter ensures that, for a sequence of i.i.d., mean-$0$ random variables\footnote{When only the distribution of the $X_j$'s is considered, the subscript $j$ is omitted for simplicity.} $X_j\in\RR$ and $Y_m:=m^{-1/2}\sum_{j=1}^mX_j$ for any $m\in\ZZ^+$, there is an absolute constant $C$ s.t.
\begin{equation}\label{rio}
\wass_2\lb\prb_m,~\N(0,\var X)\rb\lq C\lb\frac{\ex|X|^4}{\var X}\rb^\frac{1}{2}m^{-\frac{1}{2}},
\end{equation}
where $\prb_m$ denotes the distribution of $Y_m$ and $\wass_p(\cdot,\cdot)$ is the $p$-Vaserstein (or ``Wasserstein") distance. For probability measures $\prb,\qrb$ on $\RR^q$, such a distance is defined by
\begin{equation*}
\wass_p(\prb,\qrb):=\inf_{\pi\in\Pi(\prb,\qrb)}\lb\int_{\RR^q\times\RR^q}|x-y|^p\pi(\td x,\td y)\rb^\frac{1}{p},
\end{equation*}
where $\Pi(\prb,\qrb)$ is the set of all joint probability measures on $\RR^q\times\RR^q$ with marginal laws $\prb$ and $\qrb$. Rio \cite{rio2009uppboumindiscenlimthe} (Theorem 4.1) in fact only assumed the independence of $\{X_j\}$, but regarding central limit approximations and the simulation of {\levy} processes one only considers the i.i.d. case. The constant $C$ in \eqref{rio} would vary in $p$ for a bound in $\wass_p$ and is later optimised in \cite{rio2011AsyConMinDisCenLimThe}. Apart from the restriction $q=1$, Rio's effective bounds only hold for $p\lq4$. But this has been improved by Bobkov \cite{BerBouEdgExpCenLimTheTraDis} (Theorem 1.1), allowing the $\wass_p$-convergence of order $O(m^{-1/2})$ for any $p\gq1$.

The dimensional restriction in Rio and Bobkov's results comes from the fact that when $q=1$, for $p\gq1$ the $\wass_p$ distance between two probability measures $\prb,\qrb$ on $\RR$ is explicitly given (see Theorem 2.18 and Remarks 2.19 in \cite{villani2003TopOptTra}):
\begin{equation}\label{wass_1d}
\wass_p(\prb,\mathds{Q})=\lb\int_0^1\lv F^{-1}(t)-G^{-1}(t)\rv^p\td t\rb^\frac{1}{p},
\end{equation}
where $p\gq1,~F,G$ are the c.d.f's of $\prb,\mathds{Q}$, and $F^{-1},G^{-1}$ are their generalised inverses, respectively. For $p=1$ there is a further equality $\wass_1(\prb,\mathds{Q})=\int_{\RR}|F(x)-G(x)|\td x$; in general there is no explicit formula for $q\gq2$. However, if two probability distributions $\prb$ and $\mathds{Q}$ on $\RR^q$ have densities $f$ and $g$, respectively, instead of the precise formula \eqref{wass_1d} there is the inequality
\begin{equation}\label{crude_bound}
\wass_p(\prb,\mathds{Q})\lq C_p\lb\int_{\RR^q}|x|^p|f(x)-g(x)|\td x\rb^\frac{1}{p},
\end{equation}
for all $p\gq1$, as a variant of Proposition 7.10 in \cite{villani2003TopOptTra}.

This article presents an attempt to handle the normal approximation for the small jumps \eqref{small_jumps} for $q\gq2$ using the bound \eqref{crude_bound}, and give a positive answer to Fournier's question.

Davie \cite{davie2015Polpernordis} sketched an asymptotic approach via Edgeworth expansion of the density of $Y_m$, and proved (as a corollary to Proposition 2 therein) the rate $O(m^{-1/2})$ under the assumption that all moments of $X$ are bounded. Moreover, he in fact showed a coupling between $Y_m$ and the normal distribution perturbed by polynomials. Section 1 of this article basically follows Davie's approach, but expounds detailed calculations and specify the range of $p$ and precisely how many moments of $X$ are needed - see Theorem \ref{CLT} below.

The rate of convergence for the multi-dimensional central limit theorem has been studied using different methods. A strong result by Zaitsev (summarised as Theorem 2 in \cite{EstimatesStrongApproximationMultidimensionalCentralLimitTheorem} and proved as Theorem 1.3 in \cite{MulVerResKomMajTusVecFinExpMom}) gives a sharp Chernoff-type bound, and by Chebyshev's inequality the central limit theorem follows in a stronger sense: for independent $\{X_j\}$ each having identity covariance and independent standard Gaussian $\{\xi_j\}$ with partial sums $\Upsilon_m:=m^{-1/2}\sum_{j=1}^m\xi_j$, if the law of each $X_j$ satisfies certain analyticity conditions (see the definition of the class $\mathcal{A}_q(\tau)$ in \cite{EstimatesStrongApproximationMultidimensionalCentralLimitTheorem}), then the distance $\max_{k\lq m}\lv Y_k-\Upsilon_k\rv$ is of order $O(m^{-1/2}\log m)$ in probability. The logarithmic factor emerges because the method is based on the dyadic approximation by Koml{\'o}s, Major and Tusn{\'a}dy (KMT) \cite{KomlosMajorTusnady}. The KMT method is much stronger than the usual central limit theorem since it considers the simultaneous approximation between $Y_1,Y_2,\cdots,Y_m$ and $\Upsilon_1,\Upsilon_2,\cdots,\Upsilon_m$. Einmahl \cite{einmahl1989ExtResKomMajTustoMulCas} generalised the original KMT method to the multi-dimensional case, and Zaitsev's theorem \cite{MulVerResKomMajTusVecFinExpMom} is an improved version of that, albeit it requires the local existence of the moment generating function.

Since the central limit theorem only concerns the coupling between $Y_m$ and $\Upsilon_m$, one should expect the $\log m$ factor to be removed as in the one-dimensional result of Rio. This has indeed been achieved by Bobkov \cite{bobkov2013RatConEdgExpEntCenLimThe} (Theorem 6.1) under the assumption that $\ex|X|^5<\infty$; given only $\ex|X|^4<\infty$, his result is weakened to $O(m^{-1/2}(\log m)^{q/4-1})$. It is worth mentioning here that, shortly after this article had been submitted, using Stein's method Bonis \cite{bonis2016RatCenLimTheDifAppviaSteMet} (Theorem 8) managed to achieve the optimal rate $O(m^{-1/2})$ given only $\ex|X|^4<\infty$, which is a significant improvement. However, both approaches only work for $p=2$ since their arguments rely on some entropic transport inequalities for the $\wass_2$ distance. In this special case (normal approximation for $Y_m$ in $\wass_2$) the result derived in this article is not optimal, as it requires $\ex|X|^{4+\tau}<\infty$ for some $\tau\in(0,1)$ and Cram{\'e}r's condition $\varlimsup_{|s|\to\infty}|\ex\exp(isX)|<1$. 

Nevertheless, given that $\ex|X|^{6+\tau}<\infty$ and Cram{\'e}r's condition, the result here would give a coupling for $Y_m$ of order $O(m^{-1})$ in $\wass_p$ for a positive even ingeter $p$, if one perturbs the normal distribution with a cubic Edgeworth polynomial. The Edgeworth expansion is used by Bobkov \cite{BerBouEdgExpCenLimTheTraDis} (Corollary 9.2) in the one-dimensional case for higher-order approximations for $Y_m$, but in return Cram\'er's condition and some higher moments are needed. Theorem \ref{CLT} here can be regarded as a generalisation of that.

In Section 2, the central limit bound in $\wass_p$ is applied to the normal approximation for the small jumps \eqref{small_jumps}. This is done by viewing $Z_t^\epsilon$ as a compound Poisson process, assuming Cram{\'e}r's condition and that the L\'evy measure $\nu$ is sufficiently singular at $0$ (Theorem \ref{jump_normal_approx}). A desired coupling $\wass_p(Z_t^\epsilon,\sqrt{t}\N(0,\Sigma_\epsilon))=O(\epsilon)$ is then achieved for $t=\epsilon$ and $\Sigma_\epsilon=\int_{0<|z|\lq\epsilon}zz^\top\nu(\td z)$, which covers the case \ref{alpha-stable}. However, those assumptions can all be removed if one compromises for a suboptimal rate, as is proved in the appendices of Godinho's paper \cite{AsyGraColSpaHomBolEquSofCouPot} (Proposition A.2), where only bounded jumps are considered. Again, there is a logarithmic factor because the proof directly uses the aforementioned result of Zaitsev.

Finally, the significance of using the Vaserstein distances instead of other ones is that, when generating numerical approximations for an SDE, the convergence in $\wass_p$ is equivalent to the usual strong $L^p$-convergence. The reader is referred to the last section of \cite{davie2014Patappstodifequusicou} for a discussion on the contexts where such a substitution holds or fails. Unlike some of the results therein, the method to be introduced here is applicable to the simulation of stochastic flows defined by a {\levy} SDE, since it only aims at a coupling for the increments $Z_{t_k}-Z_{t_{k-1}}$.

Throughout this article the generic positive constants $C_\cdot$ and $c_\cdot$ may change their values, with subscripts indicating their dependence of parameters. The notations $\lesssim$ and $\gtrsim$ indicate inequalities that hold with a factor $C_q$. The notation $\xi_\Sigma$ always stands for an $\RR^q$-random variable following $\N(0,\Sigma)$. The symbol $|\cdot|$, depending on the object it acts on, stands for the modulus of vectors on $\RR^q$, the absolute value for scalars, and the $1$-norm of multi-indices on $\NN^q$. For any $\rho\in\NN^q$, it would be convenient to introduce the notation $|\rho|_\ast:=\rho_1+2\rho_2+\cdots+q\rho_q$. In the context of matrices, $I$ stands for the identity matrix and $\|\cdot\|$ denotes any matrix norm.

\section{A Coupling for the Central Limit Theorem}

This section follows Davie's asymptotic approach via Edgeworth expansion briefly sketched in \cite{davie2015Polpernordis}, and elaborates each step rigorously. The goal is to achieve a good $\wass_p$ bound using \eqref{crude_bound}, and for that one may first approximate the Fourier transform.

\subsection{Asymptotic Estimates of the Characteristic Function}

Let $q\in\ZZ^+$ and $\{X_j\}_{j\gq1}$ be a sequence of i.i.d. $\RR^q$-random variables with mean $0$, covariance $\Sigma$ and characteristic function $\chi$, and define the weighted sum $Y_m=m^{-1/2}\sum_{i=1}^mX_j,~m\in\ZZ^+$. Denote by $\psi_m$ and $\prb_m$ the characteristic function and distribution of $Y_m$, respectively. Then one has asymptotic expansion (using the multi-index $\alpha\in\NN^q$):
\begin{equation*}
\log\chi(s)\sim-\frac{1}{2}s\cdot\Sigma s+\sum_{|\alpha|\gq3}\frac{i^{|\alpha|}}{\alpha!}\mu_\alpha s^\alpha,
\end{equation*}
where $\mu_\alpha=\mu_\alpha(X)=i^{-|\alpha|}\partial^\alpha\log\chi(0)$ is the $\alpha$-th cumulant of $X$. This gives a formal expansion for $\log\psi_m(z)=m\log\chi(m^{-1/2}z)\sim-\frac{1}{2}z\cdot\Sigma z+\sum_{|\alpha|\gq3}\frac{i^{|\alpha|}}{\alpha!}m^{1-|\alpha|/2}\mu_\alpha z^\alpha$, implying that
\begin{equation}\label{psi_expand}
\psi_m(z)\sim e^{-\frac{1}{2}z\cdot\Sigma z}\lb1+\sum_{k=1}^\infty m^{-\frac{k}{2}}P_k(z)\rb,
\end{equation}
where $P_k(z)$ is a polynomial whose monomials have highest degree $3k$ and lowest degree $k+2$, with coefficients bounded by $C_k(\ex|X|^{k+2})^k$ - see Lemma 7.1 in \cite{bhattacharya1976NorAppAsyExp}. The inverse Fourier transform of \eqref{psi_expand} gives the Edgeworth expansion for the density $f_m$ of $Y_m$, if it exists. Detailed derivation for $q=1$ can also be found in \cite{petrov1975SumIndRanVar} (Chapter VI).

In this section the shorthand notations $\eps:=m^{-1/2},~\P_{\eps,r}:=1+\sum_{k=1}^r\eps^kP_k,~\forall r\in\ZZ^+$, and $\P_\eps:=\P_{\eps,\infty}$ are used, and $\eps$ and $m$ may be frequently interchanged. Denote by $\lambda_1\lq\cdots\lq\lambda_q$ the eigenvalues of $\Sigma$, and assume $\lambda_1\lq1\lq\lambda_q$ without loss of generality. Furthermore, $\forall M>0$ denote $\kappa_M:=1\vee\ex|X|^M$, then $\kappa_M^{1/M}$ increases in $M$ by H\"older's inequality, and so does $\kappa_M$. By Lemma 6.3 in \cite{bhattacharya1976NorAppAsyExp}, $|\mu_\alpha|\lq C_\alpha\kappa_{|\alpha|},~\forall\alpha\in\NN^q$.

\begin{lemma}\label{char_estimate_1}
Suppose $\Sigma$ is non-singular and $\ex|X|^{n+\tau}<\infty$ for a fixed integer $n\gq3$ and $\tau\in(0,1)$. Let $\beta\in(0,1/3)$ and $\delta:=\min\{\lambda_1/\kappa_3,\kappa_n^{-1/n}/2\}$. Then for $|z|\lq m^{1/2}\delta,~m\in\ZZ^+,~|\psi_m(z)|\lq\exp\lb-\frac{1}{4}z\cdot\Sigma z\rb$; for $|z|\lq m^{\beta/2}$ and $m>\max\{(\kappa_3/\lambda_1)^3,~\kappa_{n+\tau}^{\max\{4,6/(n(1-3\beta))\}}\}$,
\begin{equation}\label{annulus}
\lv\psi_m(z)-e^{-\frac{1}{2}z\cdot\Sigma z}\P_{\eps,n-2}(z)\rv\lq C_{n,\tau}\kappa_{n+\tau}^{n-2}\lb|z|^{n+1}+|z|^{3(n-1)}\rb e^{-\frac{1}{4}z\cdot\Sigma z}\eps^{n-1}.
\end{equation}
\end{lemma}
\begin{proof}
First of all Taylor's theorem gives the identity
\begin{equation}\label{taylor_chi}
\chi(s)=1-\frac{1}{2}s\cdot\Sigma s+\ex\int_0^1\frac{1}{2}e^{i\theta(s\cdot X)}(1-\theta)^2(is\cdot X)^3\td\theta.
\end{equation}
Then for $|s|\lq\delta_1:=\lambda_1/\kappa_3\lq\sqrt{2/\lambda_q}$, the inequality $\log u\lq u-1,~\forall u>0$, implies that
\begin{align*}
\log|\chi(s)|\lq&\log\lb1-\frac{1}{2}s\cdot\Sigma s+\frac{1}{6}\ex|X|^3|s|^3\rb\lq-\frac{1}{2}s\cdot\Sigma s+\frac{1}{6}\delta_1\ex|X|^3|s|^2\\
\lq&-\frac{1}{2}s\cdot\Sigma s+\frac{1}{4}\lambda_1|s|^2\lq-\frac{1}{4}s\cdot\Sigma s,
\end{align*}
and the first claim $|\psi_m(z)|\lq\exp\lb-\frac{1}{4}z\cdot\Sigma z\rb$ holds for $|z|\lq m^{1/2}\delta_1$.

On the other hand, for $|s|\lq\kappa_3^{-1/3}/2\lq\lambda_q^{-1/2}/2$, from \eqref{taylor_chi} one sees that
\begin{equation*}
\mathrm{Re}\chi(s)\gq1-\frac{1}{2}\lambda_q|s|^2-\frac{1}{6}\ex|X|^3|s|^3>\frac{1}{2},
\end{equation*}
and hence the principle branch of $\log\chi(s)$ is well-defined, and $|\chi(s)|>1/2$. For fixed $n\gq3$, define, $\forall s\in\RR^q$,
\begin{equation*}
S_n(s):=\sum_{|\alpha|=2}^n\frac{i^{|\alpha|}}{\alpha!}\mu_\alpha s^\alpha,~T_n(s):=\sum_{|\alpha|=2}^n\frac{i^{|\alpha|}}{\alpha!}s^\alpha\ex X^\alpha=\sum_{j=2}^n\frac{1}{j!}\ex(is\cdot X)^j.
\end{equation*}
Then $\forall u\in\RR$, using the inequality $|e^{iu}-1|\lq2\wedge|u|\lq2^{1-\tau}|u|^\tau,~\forall\tau\in(0,1)$, and the identity
\begin{equation*}
e^{iu}=\sum_{k=0}^n\frac{(iu)^k}{k!}+\frac{i^n}{(n-1)!}\int_0^1(1-\theta)^{n-1}u^n\lb e^{i\theta u}-1\rb\td\theta,
\end{equation*}
one deduces $|\chi(s)-1-T_n(s)|\lq C_{n,\tau}\kappa_{n+\tau}|s|^{n+\tau}$ by the substitution $u=s\cdot X$. Meanwhile one can write the following expansion (with Taylor remainder $R_n(s)$):
\begin{equation*}
\log\lb1+T_n(s)\rb=\sum_{l=1}^n\frac{(-1)^{l+1}}{l}T_n^l(s)+R_n(s)=S_n(s)+\tilde{S}_n(s)+R_n(s),
\end{equation*}
where $\tilde{S}_n(s)$ is a polynomial of which each monomial has degree at least $n+1$. The fact that the first few terms agree with $S_n(s)$ is due to the relation between the cumulants $\mu_\alpha$ and the moments $\ex X^\alpha$ - see Section 6 (page 46) in \cite{bhattacharya1976NorAppAsyExp}. By the multinomial theorem, for $l=1,\cdots,n$ each monomial in $T_n^l(s)$ takes the form
\begin{equation*}
\sigma_{\rho,l}(s)=C_{n,l,\rho}\prod_{j=1}^{n-1}\lb\ex(s\cdot X)^{j+1}\rb^{\rho_j},
\end{equation*}
for some $\rho\in\NN^{n-1},~|\rho|=l$. Then the monomials $\tilde{\sigma}_{\rho,l}(s)$ of $\tilde{S}_n$ correspond to those with $\sum_{j=1}^{n-1}(j+1)\rho_j=|\rho|_\ast+l\gq n+1$. If one further chooses $\delta_2:=\kappa_n^{-1/n}/2<1$, then for $|s|\lq\delta_2$,
\begin{align*}
|\tilde{\sigma}_{\rho,l}(s)|\lq&C_{n,l}|s|^{|\rho|_\ast+l}\prod_{j=1}^{n-1}\kappa_{j+1}^{\rho_j}\lq C_{n,l}|s|^{n+1}\kappa_n^{-(|\rho|_\ast+l-(n+1))/n}\prod_{j=1}^{n-1}\kappa_{j+1}^{(|\rho|_\ast+l)/(j+1)}\\
=&C_{n,l}|s|^{n+1}\kappa_n^{(n+1)/n}\prod_{j=1}^{n-1}\lb\kappa_n^{-1/n}\kappa_{j+1}^{1/(j+1)}\rb^{|\rho|_\ast+l}\lq C_{n,l}\kappa_n^{1+1/n}|s|^{n+1},
\end{align*}
where H\"older's inequality is used in the last step. Therefore $|\tilde{S}_n(s)|\lq C_n\kappa_n^{1+1/n}|s|^{n+1}$. Also notice that, for $|s|\lq\delta_2$ and $j=2,\cdots,n$, one has $|s|^{j-1}\kappa_j\lq\kappa_n^{1/n}(\kappa_n^{-1/n}\kappa_j^{1/j})^j\lq\kappa_n^{1/n}$. This implies that $|T_n|\lq\sum_{j=2}^n|s|^{j-1}\kappa_j|s|/j!\lq(e-2)\kappa_n^{1/n}|s|/2<1/2$, and that $|1+\theta T_n(x)|\gq1/2$ for any $\forall\theta\in[0,1]$. Therefore
\begin{equation}
|R_n(s)|\lq\int_0^1(1-\theta)^n\lv\frac{T_n(s)}{1+\theta T_n(s)}\rv^{n+1}\td\theta\lq C_n|T_n(s)|^{n+1}\lq C_n\kappa_n^{1+1/n}|s|^{n+1}.
\end{equation}
Thus $|\log(1+T_n(s))-S_n(s)|\lq C_n\kappa_n^{1+1/n}|s|^{n+1}$. Since $|\chi(s)|\wedge|1+T_n(s)|\gq1/2$ for $|s|<\delta_2$, the triangle inequality implies that
\begin{equation*}
\lv\log\chi(s)-S_n(s)\rv\lq2\lv\chi(s)-1-T_n(s)\rv+\lv\log(1+T_n(s))-S_n(s)\rv\lq C_{n,\tau}\kappa_{n+\tau}^{1+1/n}|s|^{n+1}.
\end{equation*}
%Note that\footnote{Subscripts here denote the components of the vector $X=(X_1,\cdots,X_q)$.} $\forall\alpha\in\NN^q,~u\in\RR^q$, by Jensen's inequality $\lv\partial^\alpha\chi(u)\rv=\lv\ex X_1^{\alpha_1}\cdots X_q^{\alpha_q}e^{iu\cdot X}\rv\lq\ex|X^\alpha|\lq q\kappa_{|\alpha|}$. and $|s|<\delta_2$,
%\begin{align*}
%&\lv\log\chi(s)+\frac{1}{2}s\cdot\Sigma s-\sum_{|\alpha|=3}^n\frac{i^{|\alpha|}}{\alpha!}\mu_\alpha s^\alpha\rv=\lv\sum_{|\alpha|=n+1}\frac{i^{|\alpha|}}{\alpha!}\int_0^1(1-\theta)^{|\alpha|-1}s^\alpha\partial^\alpha\log\chi(\theta s)\td\theta\rv\\
%&\lq C_{q,n}\sum_{|\alpha|=n+1}\frac{1}{\alpha!}|s|^{|\alpha|}\int_0^1(1-\theta)^{|\alpha|-1}\frac{p_{\alpha}(X)}{|\chi(\theta s)|^{|\alpha|}}\td\theta<2C_{q,n}\kappa_{n+1}|s|^{n+1},\num\label{0}
%\end{align*}
%where $p_{\alpha}(X)=\sum_\rho\ex|X^{\rho^{1}}|\cdots\ex|X^{\rho^{|\alpha|}}|$, the summation being over all partitions of $\alpha=\rho^{1}+\cdots+\rho^{|\alpha|},~\rho^j\in\NN^q$, and the last step follows from H\"older's inequality. From this one also sees that $|\mu_\alpha|=|\partial^\alpha\log\chi(0)|\lq C_{|\alpha|,q}\ex|X|^{|\alpha|}$.

Returning to $\psi_m$, as $\log\psi_m(z)=\eps^{-2}\log\chi(\eps z)$, from the estimate above one has
\begin{equation}\label{1}
\lv\log\psi_m(z)-\eps^{-2}S_n(\eps z)\rv\lq C_{n,\tau}\eps^{n-1}|z|^{n+1}\kappa_{n+\tau}^{1+1/n}.
\end{equation}
Moreover, writing $U_n(z):=\frac{1}{2}z\cdot\Sigma z+\eps^{-2}S_n(\eps z)$, one can apply Taylor's theorem again to the exponential $\exp(U_n(z))$ (recall the notation $\P_{\eps,\cdot}$):
\begin{align*}
\exp\lb\sum_{|\alpha|=3}^n\frac{i^{|\alpha|}}{\alpha!}\eps^{|\alpha|-2}\mu_\alpha z^\alpha\rb=&1+U_n(z)+\frac{1}{2!}U_n^2(z)+\cdots+\frac{1}{(n-2)!}U_n^{n-2}(z)+V(z)\\
=&1+\P_{\eps,n-2}(z)+\tilde{P}(z)+V(z),
\end{align*}
where $\tilde{P}(z)=0$ for $n=3$ (i.e. $P_1(z)=U_3(z)$ contains all the cubic terms) and otherwise a polynomial of degree $n(n-2)$ with complex coefficients that contain products of the cumulants $\mu_\alpha$ up to $|\alpha|=n$ and powers of $\eps$ at least $n-1$; the Taylor remainder $V(z)$ is given by
\begin{equation*}
V(z)=\frac{1}{(n-2)!}\int_0^1(1-\theta)^{n-2}U_n^{n-1}(z)e^{\theta U_n(z)}\td\theta.
\end{equation*}
For $|z|\lq m^{1/6}=\eps^{-1/3}$, one claims the following bound:
\begin{equation*}
\lv\tilde{P}(z)\rv\lq C_n\kappa_n^{n-2}\eps^{n-1}(|z|^{n+3}+|z|^{3(n-1)}).
\end{equation*}
This can be seen by checking the powers of $\eps$ and $z$ in each $U_n^l(z),~l=1,\cdots,n-2$. For each $l$, the multinomial theorem gives (with multi-indices $\rho\in\NN^{n-2},~\alpha\in\NN^q$)
\begin{equation*}
U_n^l(z)=(-1)^l\sum_{|\rho|=l}\binom{l}{\rho}(i\eps)^{|\rho|_\ast}\prod_{j=1}^{n-2}\lb\sum_{|\alpha|=j+2}\frac{1}{\alpha!}\mu_\alpha z^\alpha\rb^{\rho_j}.
\end{equation*}
Then each monomial of $U_n^l(z)$ is bounded by $C_{n,l}\kappa_n^l\eps^{|\rho|_\ast}|z|^{|\rho|_\ast+2l}$, and the monomials $\tilde{p}_{\rho,l}(z)$ of $\tilde{P}(z)$ correspond to those with $|\rho|_\ast\gq n-1$ and $l\gq2$. When $|\rho|_\ast+2l\lq3(n-1)$ the claim follows immediately from interpolating the powers of $|z|$; when $|\rho|_\ast+2l>3(n-1)$, note that $|\rho|_\ast>|\rho|=l$, and so for $|z|\lq\eps^{-1/3}$,
\begin{equation*}
|\tilde{p}_{\rho,l}(z)|\lq C_{n,l}\kappa_n^l\eps^{\frac{2}{3}(|\rho|_\ast-l)+n-1}|z|^{3(n-1)}\lq C_n\kappa_n^{n-2}\eps^{n-1}|z|^{3(n-1)}.
\end{equation*}
Regarding the Taylor remainder $V(z)$, notice that for $|z|\lq\eps^{-\beta},~\forall\beta\in(0,1/3)$, and $\eps<\kappa_n^{-1}$,
\begin{align*}
|U_n(z)|\lq&\sum_{j=1}^{n-2}\eps^j|z|^{j+2}\kappa_{j+2}\lq\sum_{j=1}^{n-2}\eps^{j-\beta(j-1)}|z|^3\kappa_{j+2}\lq\sum_{j=1}^{n-2}\eps^{\frac{2}{3}(j-1)}\kappa_n^{(j+2)/n}\eps|z|^3\\
\lq&\sum_{j=1}^{n-2}\kappa_n^{\frac{3}{n}+\lb\frac{1}{n}-\frac{2}{3}\rb(j-1)}\eps|z|^3\lq(n-2)\kappa_n^{3/n}\eps|z|^3,
\end{align*}
and furthermore $|U_n(z)|\lq(n-2)\kappa_n^{3/n}\eps^{1-3\beta}$. Thus one arrives at
\begin{equation*}
|V(z)|\lq C_n\kappa_n^3\exp\lb(n-2)\eps^{1-3\beta}\kappa_n^{3/n}\rb\eps^{n-1}|z|^{3(n-1)}.
\end{equation*}
Combining with \eqref{1} one deduces, for $|z|\lq\eps^{-\beta}$,
\begin{align*}
&\lv\psi_m(z)-e^{-\frac{1}{2}z\cdot\Sigma z}\mathcal{P}_\eps^{(n-2)}(z)\rv\\
&\lq\lv e^{\log\psi_m(z)}-e^{-\frac{1}{2}z\cdot\Sigma z+U_n(z)}\rv+\lv e^{-\frac{1}{2}z\cdot\Sigma z+U_n(z)}-e^{-\frac{1}{2}z\cdot\Sigma z}\P_{\eps,n-2}(z)\rv\\
&\lq|\psi_m(z)|\lv1-\exp\lb-\log\psi_m(z)-\frac{1}{2}z\cdot\Sigma z+U_n(z)\rb\rv+e^{-\frac{1}{2}z\cdot\Sigma z}\lb|\tilde{P}(z)|+|V(z)|\rb\\
&\lq C_{n,\tau}|\psi_m(z)|\exp\lb\eps^{2(n-2)/3}\kappa_{n+\tau}^{1+1/n}\rb\eps^{n-1}|z|^{n+1}\kappa_{n+\tau}^{1+1/n}\\
&\quad+C_n\kappa_n^{n-2}\exp\lb(n-2)\eps^{1-3\beta}\kappa_n^{3/n}\rb\eps^{n-1}(|z|^{n+3}+|z|^{3(n-1)})e^{-\frac{1}{2}z\cdot\Sigma z},
\end{align*}
where in the last step the inequality $|1-e^u|\lq e^{|u|}|u|,~\forall u\in\CC$, is used for the first term.

Now with $\delta:=\delta_1\wedge\delta_2$ fixed, for $m$ large one has $m^{\beta/2}<m^{1/2}\delta$. Also, for fixed $\beta\in(0,1/3)$ and $\tau\in(0,1)$, one may further choose $m>\kappa_{n+\tau}^{3(1+1/n)/(n-2)}\vee\kappa_n^{6/(n(1-3\beta))}$ s.t. the exponents in coefficients above are bounded by $1$. This is satisfied when $m>\kappa_{n+\tau}^{\max\{4,6/(n(1-3\beta))\}}$. For $m>\delta^{-3}>\delta^{2/(\beta-1)}$ the first claim still holds, and so the second claim follows.
\end{proof}

In order to bound the integral of the left-hand side term in \eqref{annulus} over all of $\RR^q$, one may assume \textbf{Cram{\'e}r's condition}:
\begin{equation*}
\varlimsup_{|s|\to\infty}|\chi(s)|<1,
\end{equation*}
or equivalently, there exist $\rho>0$ and $\gamma\in(0,1)$ s.t. the following condition holds:

\begin{assumption}{$\mathbf{CC}(\rho,\gamma)$}\label{cramer_interp}
$|\chi(s)|\lq\gamma,~\forall|s|\gq\rho$.
\end{assumption}

As explained in \cite{bhattacharya1976NorAppAsyExp} (page 207), if $\chi$ satisfies Cram{\'e}r's condition, then $|\chi(s)|<1,~\forall s\neq0$; it is satisfied when $X$ has a density by the Riemann-Lebesgue theorem. Discrete distributions are excluded, but some singular and yet non-lattice distributions are also allowed, such as the distribution on the Cantor middle-third set that gives mass $2^{-j}$ to each interval on the $j$-th level.

Given the $X_j$'s satisfying Cram\'er's condition, the following lemma shows that it is also satisfied for the weighted sum $Y_m$.
\begin{lemma}\label{cramer}
Let $\chi$ satisfy \ref{cramer_interp} with $\rho,\gamma$ explicitly known and $\delta\in(0,\rho\wedge1)$. Then $\exists\bar{\gamma}=\bar{\gamma}(\rho,\gamma,\delta)\in(0,1)$ s.t. $|\psi_m(z)|<\bar{\gamma}^m$ for $|z|>m^{1/2}\delta$. 
\end{lemma}
\begin{proof}
Let $N\in\ZZ^+$ and write $\chi(Ns)=|\chi(Ns)|e^{i\theta_1},~\chi(s)=|\chi(s)|e^{i\theta_0}$, where $\theta_1,~\theta_0$ depend on $s$. Then, with $F$ being the distribution of $X$, one gets $\int_{\RR^q}\sin(s\cdot x-\theta_0)F(\td x)=0$ and
\begin{align*}
1-|\chi(s)|=&\int_{\RR^q}\lb1-\cos(s\cdot x-\theta_0)\rb F(\td x)=\int_{\RR^q}2\sin^2\frac{1}{2}(s\cdot x-\theta_0)F(\td x)\\
\gq&\frac{1}{N^2}\int_{\RR^q}2\sin^2\frac{N}{2}(s\cdot x-\theta_0)F(\td x)=\frac{1}{N^2}\int_{\RR^q}\lb1-\cos(Ns\cdot x-N\theta_0)\rb F(\td x),
\end{align*}
where the inequality $|\sin(N\phi)|\lq N|\sin\phi|,~\forall N\in\NN,~\phi\in\RR$, is used. Meanwhile,
\begin{equation*}
|\chi(Ns)|=e^{-i\theta_1}\int_{\RR^q}e^{iNs\cdot x}F(\td x)=e^{i(N\theta_0-\theta_1)}\int_{\RR^q}e^{i(Ns\cdot x-N\theta_0)}F(\td x),
\end{equation*}
which implies
\begin{align*}
1-|\chi(s)|\gq&\frac{1}{N^2}-\frac{1}{N^2}\textrm{Re}\int_{\RR^q}e^{i(Ns\cdot x-N\theta_0)}F(\td x)\nonumber\\
\gq&\frac{1}{N^2}-\frac{1}{N^2}\lv\int_{\RR^q}e^{i(Ns\cdot x-N\theta_0)}F(\td x)\rv=\frac{1}{N^2}-\frac{1}{N^2}|\chi(Ns)|.
\end{align*}
Choose $N=[(\rho+1)/\delta]>\rho/\delta$, then $|\chi(s)|\lq1-(1-\gamma)\delta^2/(\rho+1)^2=:\bar{\gamma}$ for $\delta<|s|<\rho$. Clearly $\bar{\gamma}\gq\gamma$, and $|\psi_m(z)|=|\chi(m^{-1/2}z)|^m<\bar{\gamma}^m<1$ for $|z|>m^{1/2}\delta$.
\end{proof}

From now on the following bounds will be frequently used: $\forall M,c>0$,
\begin{equation}\label{bound_1}
\int_{\RR^q}|x|^Me^{-cx\cdot\Sigma x}\td x=\int_{\RR^q}\lv\Sigma^{-\frac{1}{2}}y\rv^Me^{-c|y|^2}\det\lb\Sigma^{-\frac{1}{2}}\rb\td y\lq C_{q,c,M}(\det\Sigma)^{-\frac{1}{2}}\lambda_1^{-\frac{M}{2}},
\end{equation}
and
\begin{equation}\label{bound_2}
\int_{\RR^q}|x|^M\phi_\Sigma(x)\td x=C_q\int_{\RR^q}\lv\Sigma^\frac{1}{2}y\rv^Me^{-\frac{1}{2}|y|^2}\td y\lq C_{q,M}\lambda_q^\frac{M}{2},
\end{equation}
where the inverse and the square root of $\Sigma$ are well-defined since it is positive definite. 

Although Cram\'er's condition gives some restriction on the law of $X$, it does not require the smoothness or the existence of the density $f_m$ of $Y_m$. In order to see how close the law of $Y_m$ is to the perturbed normal distributions from polynomial expansions, one may use a smoothing argument. Let $\tilde{f}_m$ and $\tilde{\psi}_m$ be the density and characteristic function of the mollified measure $\prb_m\ast\theta_m$, where $\theta_m$ is a measure with smooth density, still denoted by $\theta_m$ or $\theta_\eps$:
\begin{equation}\label{theta_def}
\theta_\eps(x)=\eps^{-q(n+1)}h(\eps^{-n-1}x),
\end{equation}
for some function $0\lq h\in C_0^\infty(\RR^q)$ supported on the open unit ball and $\int_{\RR^q}h(x)\td x=1$. Thus $\theta_\eps$ is a probability density supported on $\{|x|<\eps^{n+1}\}$. Write $\hat{h}$ and $\hat{\theta}_\eps$ as their respective Fourier transforms.

\begin{proposition}\label{prop}
Under the assumptions in Lemma \ref{char_estimate_1} and Lemma \ref{cramer}, for any integer $n\gq3,~\tau\in(0,1),~\beta\in(0,1/3)$ and $m$ sufficiently large, it holds true that
\begin{equation*}
\int_{\RR^q}\lv\tilde{\psi}_m(z)-e^{-\frac{1}{2}z\cdot\Sigma z}\P_{\eps,n-2}(z)\rv\td z\lq C_{q,n,\tau}(\det\Sigma)^{-\frac{1}{2}}\lambda_1^{-\frac{n-1}{2\beta}}\kappa_{n+\tau}^{n-2}\eps^{n-1}.
\end{equation*}
\end{proposition}
\begin{proof}
Note that $\tilde{\psi}_m=\psi_m\hat{\theta}_\eps$, and for $|z|\lq m^{1/2}\delta$,
\begin{align*}
\lv\tilde{\psi}_m(z)-\psi_m(z)\rv=&|\psi_m(z)|\lv\hat{\theta}_\eps(z)-1\rv\lq|\psi_m(z)|\int_{|x|<\eps^{n+1}}\lv e^{iz\cdot x}-1\rv\theta_\eps(x)\td x\\
\lq&|\psi_m(z)||z|\eps^{n+1}\lq \eps^{n+1}|z|e^{-\frac{1}{4}z\cdot\Sigma z},
\end{align*}
and hence by Lemma \ref{char_estimate_1} and triangle inequality,
\begin{equation*}
\lv\tilde{\psi}_m(z)-e^{-\frac{1}{2}z\cdot\Sigma z}\mathcal{P}_\eps^{(n-2)}(z)\rv\lq C_{n,\tau}\eps^{n-1}\kappa_{n+\tau}^{n-2}\lb|z|^{n+1}+|z|^{3(n-1)}\rb e^{-\frac{1}{4}z\cdot\Sigma z},
\end{equation*}
for $|z|\lq m^{\beta/2}$. Also for all $z\in\RR^q$,
\begin{align}
\lv\hat{\theta}_\eps(z)\rv=&\lv\int_{|x|<\eps^{n+1}}e^{iz\cdot x}\theta_\eps(x)\td x\rv=\lv\int_{|x|<\eps^{n+1}}e^{iz\cdot x}\eps^{-q(n+1)}h(\eps^{-n-1}x)\td x\rv\nonumber\\
=&\lv\int_{|y|<1}e^{i\eps^{n+1}z\cdot y}h(y)\td y\rv=\lv\hat{h}(\eps^{n+1}z)\rv\lq C_q\eps^{-K(n+1)}|z|^{-K},\label{K}
\end{align}
for any $K>0$, since $h\in C_0^\infty(\RR^q)$ with all the derivatives in $L^1(\RR^q)$. One may choose $K=q+1$ for convenience and $|\tilde{\psi}_m(z)|\lq\bar{\gamma}^m\min\{1,C_q\eps^{-(q+1)(n+1)}|z|^{-q-1}\}$ for $|z|>m^{1/2}\delta$. For $|z|\lq m^{1/2}\delta$ one still has $|\tilde{\psi}_m(z)|\lq\exp(-\frac{1}{4}z\cdot\Sigma z)$.

Given all the estimates for $\tilde{\psi}_m(z)$ on different domains, one can split the integral in question into three parts:
\begin{align*}
\tilde{I}:=&\int_{\RR^q}\lv\tilde{\psi}_m(z)-e^{-\frac{1}{2}z\cdot\Sigma z}\mathcal{P}_\eps^{(n-2)}(z)\rv\td z\\
=&\lb\int_{|z|\lq m^{\beta/2}}+\int_{m^{\beta/2}<|z|\lq m^{1/2}\delta}+\int_{|z|> m^{1/2}\delta}\rb\lv\tilde{\psi}_m(z)-e^{-\frac{1}{2}z\cdot\Sigma z}\mathcal{P}_\eps^{(n-2)}(z)\rv\td z.
\end{align*}
Then by virtue of Lemma \ref{char_estimate_1}, Lemma \ref{cramer},
\begin{align*}
\tilde{I}&\lq C_{n,\tau}\kappa_{n+\tau}^{n-2}\eps^{n-1}\int_{|z|\lq m^{\beta/2}}\lb|z|^{n+1}+|z|^{3(n-1)}\rb e^{-\frac{1}{4}z\cdot\Sigma z}\td z+\int_{m^{\beta/2}<|z|\lq m^\frac{1}{2}\delta}e^{-\frac{1}{4}z\cdot\Sigma z}\td z\\
&+\int_{|z|>m^{1/2}\delta}\bar{\gamma}^m\lb1\wedge C_q\eps^{-(q+1)(n+1)}|z|^{-q-1}\rb\td z+\int_{|z|>m^{\beta/2}}e^{-\frac{1}{2}z\cdot\Sigma z}|\P_{\eps,n-2}(z)|\td z.
\end{align*}
Use \eqref{bound_1} for the first integral, combine the second and the fourth, and split the third into the set where $|z|$ is large and its complement to get ($\Lambda$ denotes the Lebesgue measure on $\RR^q$)
\begin{align*}
\tilde{I}\lq&C_{q,n,\tau}(\det\Sigma)^{-\frac{1}{2}}\lambda_1^{-\frac{3}{2}(n-1)}\kappa_{n+\tau}^{n-2}\eps^{n-1}+\bar{\gamma}^m\Lambda\lb\{|z|\lq C_q\eps^{-n-1}\}\rb\\
&+C_q\bar{\gamma}^m\eps^{-(q+1)(n+1)}\int_{|z|>C_q\eps^{-n-1}}|z|^{-q-1}\td z+2\int_{|z|>m^{\beta/2}}e^{-\frac{1}{4}z\cdot\Sigma z}\lb1+\sum_{k=1}^{n-2}\eps^k\lv P_k(z)\rv\rb\td z\\
\lq&C_{q,n,\tau}(\det\Sigma)^{-\frac{1}{2}}\lambda_1^{-\frac{3}{2}(n-1)}\kappa_{n+\tau}^{n-2}\eps^{n-1}+C_q\bar{\gamma}^m\eps^{-(q+1)(n+1)}\\
&+C_{q,n}\int_{|z|>m^{\beta/2}}e^{-\frac{1}{4}z\cdot\Sigma z}\kappa_n^{n-2}\lb1+\sum_{k=1}^{n-2}\eps^k|z|^{3k}\rb\td z.
\end{align*}
The second term can be absorbed by the first term if $m$ is sufficiently large s.t. it satisfies the criterion of Lemma \ref{char_estimate_1} and that
\begin{equation}\label{large_m}
\bar{\gamma}^mm^{\frac{1}{2}(q+1)(n+1)}\lq(\det\Sigma)^{-\frac{1}{2}}\lambda_1^{-\frac{3}{2}(n-1)}\kappa_{n+\tau}^{n-2}.
\end{equation}
For the third term, notice that $|z|>1$ and that $1<\eps|z|^{1/\beta},~\forall\beta\in(0,1/3)$. Thus
\begin{equation*}
\tilde{I}\lq C_{q,n,\tau}(\det\Sigma)^{-\frac{1}{2}}\lambda_1^{-\frac{3}{2}(n-1)}\kappa_{n+\tau}^{n-2}\eps^{n-1}+C_{q,n}\kappa_n^{n-2}\int_{\RR^q}e^{-\frac{1}{4}z\cdot\Sigma z}(\eps|z|^{1/\beta})^{n-1}\td z,
\end{equation*}
and the result follows from \eqref{bound_1} again.
\end{proof}

\subsection{Perturbed Normal Distributions}

Now given Proposition \ref{prop}, one can approximate the density $\tilde{f}_m$ by the inverse Fourier transform $\F^{-1}$ of $\exp\lb-\frac{1}{2}z\cdot\Sigma z\rb\P_\eps(z)$. Define, $\forall x\in\RR^q$, the Edgeworth polynomials $\{Q_k\}$ by
\begin{equation}\label{edgeworth}
\phi_\Sigma(x)Q_k(x):=\F^{-1}\left\{e^{-\frac{1}{2}z\cdot\Sigma z}P_k(z)\right\}(x),~\forall k\in\ZZ^+,
\end{equation}
and accordingly $\Q_{\eps,r}:=1+\sum_{k=1}^r\eps^kQ_k,~\forall r\in\ZZ^+$. Then each monomial of $Q_k$ has the same degree as that of $P_k$. In fact, if $\Sigma=\mathrm{diag}(\lambda_1,\cdots,\lambda_q)$, one can explicitly show that
\begin{equation}\label{hermite}
Q_k(x)=\sum_{|\alpha|=k+2}^{3k}(-1)^{|\alpha|}b_\alpha\prod_{j=1}^q\lambda_j^{-\alpha_j/2}H_{\alpha_j}\lb\lambda_j^{-1/2}x_j\rb,
\end{equation}
where $b_\alpha=b_\alpha(\mu_\beta:~|\beta|\lq k+2)$ is the real coefficient of $(iz)^\alpha$ in $P_k(z)$ satisfying $|b_\alpha|\lq\kappa_{k+2}^k$, and $H_j$ is the Hermite polynomial of degree $j$. See \cite{petrov1975SumIndRanVar} (Chapter VI \S1) for the precise values.

\begin{remark}\label{moment_match}
Since $\exp(-\frac{1}{2}z\cdot\Sigma z)\P_{\eps,n-2}(z)$ and $\psi_m(z)$ have the same derivatives at $0$ up to order $n$, the Edgeworth sum $\phi_\Sigma\Q_{\eps,n-2}$ and $Y_m$ have the same moments up to order $n$. 
\end{remark}

For a positive-definite $q\times q$ matrix $\Sigma$, let $\mathscr{P}_\Sigma$ be the set of polynomials $S:\RR^q\to\RR$ s.t. $\int_{\RR^q}S_j(x)\phi_\Sigma(x)\td x=0$ and $\mathscr{P}_G$ be the set of polynomials $U:\RR^q\to\RR^q$ s.t. $U=\nabla u$ for some polynomial $u:\RR^q\to\RR$. Furthermore let $\mathscr{P}_\Sigma^\infty$ be the set of sequences $(S_1,S_2,\cdots),~S_j\in\mathscr{P}_\Sigma$, and $\mathscr{P}_G^\infty$ be the set of sequences $(U_1,U_2,\cdots),~U_j\in\mathscr{P}_G$.

Given polynomials $U_j:\RR^q\to\RR^q,~j=1,\cdots,k$, define $\mathbf{U}_{\eps,k}(x):=x+\sum_{j=1}^k\eps^jU_j(x),~\forall\eps>0$. Then for a $\xi_\Sigma$ following $\N(0,\Sigma)$, the random variable $\mathbf{U}_{\eps,k}(\xi_\Sigma)$ is said to have a perturbed normal distribution, whose density can be formally expressed as
\begin{equation*}
\zeta_{\eps,k}(y)=\det\lb\tD\mathbf{U}_{\eps,k}^{-1}(y)\rb\phi_\Sigma\lb\mathbf{U}_{\eps,k}^{-1}(y)\rb.
\end{equation*}
Davie \cite{davie2015Polpernordis} (Section 2) showed, using a recursive construction, that one can approximate $\zeta_{\eps,k}(y)$ by $\phi_\Sigma(y)\mathcal{S}_{\eps,l}(y):=\phi_\Sigma(y)\lb1+\sum_{j=1}^l\eps^jS_j(y)\rb$ up to order $O(\eps^{l+1})$, where for each $j\lq l,~S_j:\RR^q\to\RR$ is a polynomial uniquely determined by $U_1,\cdots,U_j$ only. Since $l$ is arbitrary, for each $k$ the polynomials $U_1,\cdots,U_k$ uniquely determine a sequence $(S_1,S_2,\cdots)$, and hence the map $\mathfrak{S}_\Sigma:(U_1,U_2\cdots)\mapsto(S_1,S_2,\cdots)$ is well-defined. Moreover, each $S_j\in\mathscr{P}_\Sigma$ by Lemma 1 in \cite{davie2015Polpernordis}.

A given sequence $(S_1,S_2,\cdots)\in\mathscr{P}_\Sigma^\infty$ can have several preimages under $\mathfrak{S}_\Sigma$. But according to Lemma 2 in \cite{davie2015Polpernordis}, if one restricts $\mathfrak{S}_\Sigma$ on $\mathscr{P}_G^\infty$ then it is a bijection\footnote{This is motivated by Brenier's transport theorem for the quadratic cost - see Theorem 2.12 in \cite{villani2003TopOptTra} for the general statement and Lemma 5 in \cite{davie2015Polpernordis} for a special case.}. As is shown in the preceding paragraphs therein, this follows from the bijectivity of the linear map
\begin{equation*}
\mathscr{L}_\Sigma:\mathscr{P}_G\to\mathscr{P}_\Sigma,~U(x)\mapsto\nabla\cdot U(x)-x\cdot\Sigma^{-1}U(x).
\end{equation*}
The preimages of the bijection $\mathfrak{S}_\Sigma$ are defined inductively in the following way: given a sequence $(S_1,S_2,\cdots)\in\mathscr{P}_\Sigma^\infty$, suppose $U_1,\cdots,U_k\in\mathscr{P}_G$ are found with $\mathfrak{S}_\Sigma(U_1,\cdots,U_k)=(S_1,\cdots,S_k,\tilde{S}_{k+1},\cdots)$, then adding an additional $U_{k+1}$ gives a different sequence
\begin{equation*}
\mathfrak{S}_\Sigma(U_1,\cdots,U_k,U_{k+1})=(S_1,\cdots,S_k,\tilde{S}_{k+1}-\mathscr{L}_\Sigma U_{k+1},\cdots).
\end{equation*}
This means that $U_{k+1}\in\mathscr{P}_G$ is determined by the equation $\tilde{S}_{k+1}-\mathscr{L}_\Sigma U_{k+1}=S_{k+1}$. Writing $U_{k+1}=\nabla u_{k+1}$, one looks for a polynomial $u_{k+1}$ that solves the Hermite-type equation
\begin{equation}\label{PDE}
-\Delta u_{k+1}(x)+x\cdot\Sigma^{-1}\nabla u_{k+1}(x)=S_{k+1}(x)-\tilde{S}_{k+1}(x),~x\in\RR^q.
\end{equation}
For the initial step set $\tilde{S}_1\equiv0$ and solve the PDE by induction on the degree of $u_1$; at each step, first compute $\tilde{S}_{k+1}$ from $u_1,\cdots,u_k$ and then solve the PDE again by induction on the degree of $u_{k+1}$ - see similar arguments presented in the proof of Lemma 1 in \cite{davie2014Patappstodifequusicou}.

The computation of $\tilde{S}_{k+1}(x)$ can be done in the following formal way. First write
\begin{equation}\label{zeta}
\phi_\Sigma(x)=\zeta_{\eps,k}\lb\mathbf{U}_{\eps,k}(x)\rb\det\lb\tD\mathbf{U}_{\eps,k}(x)\rb,
\end{equation}
by a change of variables. With $U_j=\nabla u_j,~\tD\mathbf{U}_{\eps,k}(x)=I+\sum_{j=1}^k\eps^j\tD^2u_j(x)$, and so the determinant above can be expressed as $1+\eps v_1(x)+\cdots+\eps^{qk}v_{qk}(x)$, where for each $l\lq qk,~v_l$ is the sum of $(\partial_{i_1j_1}^2u_1)^{\rho_1}\cdots(\partial_{i_kj_k}^2u_k)^{\rho_k}$ over all the second derivatives and all multi-indices $\rho\in\NN^k$ s.t. $|\rho|_\ast=l$. Then by formally writing $\zeta_{\eps,k}(y)=\phi_\Sigma(y)\tilde{\mathcal{S}}_\eps(y)$ with $y=\mathbf{U}_{\eps,k}(x)$ and $\tilde{\mathcal{S}}_\eps(y)=1+\sum_{j=1}^k\eps^jS_j(y)+\sum_{j=k+1}^\infty\eps^j\tilde{S}_j(y)$, one can rearrange \eqref{zeta} to get
\begin{align*}
&1+\eps S_1(y)+\cdots+\eps^kS_k(y)+\eps^{k+1}\tilde{S}_{k+1}(y)+\cdots\\
&=\frac{\exp\left\{\sum_{j=1}^k\eps^jx\cdot\Sigma^{-1}\nabla u_j(x)+\frac{1}{2}\sum_{j_1,j_2=1}^k\eps^{j_1+j_2}\nabla u_{j_1}(x)\cdot\Sigma^{-1}\nabla u_{j_2}(x)\right\}}{1+\eps v_1(x)+\cdots+\eps^{qk}v_{qk}(x)}\\
&=1+\eps T_1(x)+\eps^2T_2(x)+\cdots,
\end{align*}
where the series on the right-hand side is obtained by multiplying out the Maclaurin series for $e^z$ and $1/(1+z)$. Since differentiating a polynomial only changes its coefficients by a constant and reduces its degree, $|T_{k+1}(x)|\lq C_{q,k}\|\Sigma^{-1}\|^{k+1}\sum_{|\rho|_\ast=k+1}(1+|u_1(x)|)^{\rho_1}\cdots(1+|u_k(x)|)^{\rho_k}$.

On the left-hand side, each polynomial $S_j(y)$ with degree $d_j\gq1$ can be expressed as $S_j(x)+\eps w_{j,1}(x)+\cdots+\eps^{d_jk}w_{j,d_jk}(x)$ by its Taylor expansion about $x$. Since all the derivatives of $S_j(x)$ have their norms bounded by $C_{q,j}(1+|S_j(x)|)$, one has, for each $j\lq k$ and $l\lq d_jk$, that $|w_{j,l}(x)|\lq C_{q,j,l}\sum_{s=1}^l\sum_{|\rho|_\ast=l}(1+|S_j(x)|)|U_1(x)|^{\rho_1}\cdots|U_s(x)|^{\rho_s}$. Thus, by expanding out $\tilde{S}_{k+1}(y)$ in terms of $x$ and matching the coefficients of $\eps^{k+1}$ on both sides, one gets
\begin{equation}\label{formal}
\tilde{S}_{k+1}(x)=T_{k+1}(x)-w_{k,1}(x)-w_{k-1,2}(x)-\cdots-w_{2,k-1}(x)-w_{1,k}(x).
\end{equation}
Although the calculation for $\tilde{S}_{k+1}$ above is completely formal, it is equivalent to Davie's construction in \cite{davie2015Polpernordis} due to the uniqueness of the power series expansion. For a rigorous proof of such an approximation of $\zeta_{\eps,k}$, the reader is referred to Proposition 1 in \cite{davie2015Polpernordis}.

\begin{remark}\label{rotation}
The set $\mathscr{P}_G$ is invariant under orthogonal transformation: given $U(x)\in\mathscr{P}_G$ and an orthogonal matrix $A$, the polynomial $G(x)=A^{-1}U(Ax)$ also lies in $\mathscr{P}_G$.
\end{remark}
To see this, notice that if $U(x)=\nabla u(x)$ and $A$ is a $q\times q$ matrix, then $g(x):=u(Ax)$ has gradient $A^\top U(Ax)$ and so $G(x)=\nabla u(Ax)$ if $A$ is orthogonal.

The following lemma is a quantitative application of Proposition 1 in \cite{davie2015Polpernordis}.

\begin{lemma}\label{bijection}
The real polynomials $\{Q_k\}_{k=1}^\infty$ uniquely determine a sequence of polynomials $\{p_k\}_{k=1}^\infty\in\mathscr{P}_G^\infty$ s.t. $\forall r\in\ZZ^+$ and $\eps$ sufficiently small,
\begin{description}
\item[(i)]$|p_k(x)|\lq C_{q,k}\lambda_1^{-5k(k+2)}\lambda_q^{1+\frac{5}{2}k(k+2)}\kappa_{r+2}^{k^2}\lb1+|x|^{3k}\rb$ for all $k=1,\cdots,r$ and $x\in\RR^q$;
\item[(ii)]The random variable $\mathbf{p}_{\eps,r}(\xi_\Sigma):=\xi_\Sigma+\sum_{k=1}^r\eps^kp_k(\xi_\Sigma)$ has density $\zeta_{\eps,r}(x)=\phi_\Sigma(x)\Q_{\eps,r}(x)+R_{\eps,r}(x)$, where for any $M\gq1$,
\begin{equation*}
\int_{\RR^q}|x|^M|R_{\eps,r}(x)|\td x\lq C_{q,r,M}\lambda_1^{-5(r+1)(r+2)}\lambda_q^{\frac{5}{2}(r+1)(r+3)+\frac{M}{2}}\kappa_{r+2}^{(r+1)^2}\eps^{r+1}.
\end{equation*}
\end{description}
\end{lemma}
\begin{proof}
First of all, the Edgeworth polynomials $\{Q_k\}$ defined by \eqref{edgeworth} are orthogonal to $\phi_\Sigma$:
\begin{equation*}
\int_{\RR^q}\phi_\Sigma(x)Q_k(x)\td x=\widehat{\phi_\Sigma Q_k}(0)=1\cdot P_k(0)=0.
\end{equation*}
Thus $\{Q_k\}\in\mathscr{P}_\Sigma^\infty$, and hence $\{p_k\}:=\mathfrak{S}_\Sigma^{-1}(\{Q_k\})$ gives the sequence sought after; for a fixed $r,~p_1,\cdots,p_r$ are determined by $Q_1,\cdots,Q_r$ only. Moreover, if $\mathfrak{S}_\Sigma(p_1,\cdots,p_r)=(Q_1,\cdots,Q_r,\tilde{Q}_{r+1},\cdots)$, then the density $\zeta_{\eps,r}$ of $\mathbf{p}_{\eps,r}(\xi_\Sigma)$ can be approximated by the expansion $\phi_\Sigma(\Q_{\eps,r}+\eps^{r+1}\tilde{Q}_{r+1})$ according to Proposition 1 in \cite{davie2015Polpernordis}. More precisely, $\forall M\gq1$,
\begin{equation}\label{density_approx}
\int_{\RR^q}|x|^M|\zeta_{\eps,r}(x)-\phi_\Sigma(x)(\Q_{\eps,r}(x)+\eps^{r+1}\tilde{Q}_{r+1}(x))|\td x\lq C_{q,r,M}K_r^{N_r}\eps^{r+2},
\end{equation}
where $K_r$ is an upper bound for $\|\Sigma\|,\|\Sigma^{-1}\|$ and the absolute value of the coefficients of $p_1,\cdots,p_r$, and $N_r=N_r(q,M)>0$ is a constant depending on the maximum degree of $p_1,\cdots,p_r$. Then for $\eps\lq K_r^{-N_r}$ this bound can be brought down to $C_{q,r,M}\eps^{r+1}$, and it remains to find an upper bound for $\tilde{Q}_{r+1}$ to derive the estimates in question.

For all $k\lq r$, write $p_k=\nabla u_k$ where $u_k$ satisfies \eqref{PDE} with $S_k\equiv Q_k$ and $\tilde{S}_k\equiv\tilde{Q}_k$. Assume that $\Sigma$ is diagonal. Then by \eqref{hermite}, $\forall k,x$ one has $|Q_k(x)|\lq C_{q,k}\lambda_1^{-3k}\kappa_{k+2}^k(1+|x|^{3k})$. Now one can bound the polynomials $\{\tilde{Q}_k\}$ and $\{u_k\}$ inductively. For each $k\lq r-1$ suppose that (i) holds true for all $j\lq k$:
\begin{equation*}
|u_j(x)|\lq C_{q,j}\lambda_1^{-5j(j+2)}\lambda_q^{1+\frac{5}{2}j(j+2)}\kappa_{r+2}^{j^2}(1+|x|^{3j}).
\end{equation*}
From the construction of $T_{k+1}$ and $\{w_{j,l}\}$ one sees that,
\begin{align*}
|T_{k+1}(x)|\lq&C_{q,k}\|\Sigma^{-1}\|^{k+1}\lambda_1^{-5\sum^\ast j(j+2)\rho_j}\lambda_q^{k+1+\frac{5}{2}\sum^\ast j(j+2)\rho_j}\kappa_{r+2}^{\sum^\ast j^2\rho_j}\lb1+|x|^{\sum^\ast3j\rho_j}\rb\\
\lq&C_{q,k}\lambda_1^{-(k+1)(5k+11)}\lambda_q^{(k+1)\lb\frac{5}{2}k+6\rb}\kappa_{r+2}^{(k+1)^2}\lb1+|x|^{3(k+1)}\rb,\num\label{induction}\\
|w_{j,l}(x)|\lq&C_{q,j,l}\lambda_1^{-3j-5\sum^\dag s(s+2)\rho_s}\lambda_q^{l+\frac{5}{2}\sum^\dag s(s+2)\rho_s}\kappa_{r+2}^{j+\sum^\dag s^2\rho_s}\lb1+|x|^{3j+\sum^\dag3s\rho_s}\rb\\
\lq&C_{q,j,l}\lambda_1^{-3j-5l(l+2)}\lambda_q^{l+\frac{5}{2}l(l+2)}\kappa_{r+2}^{j+l(l+1)}\lb1+|x|^{3(j+l)}\rb,
\end{align*}
where $\sum^\ast$ denotes the summation over $j=1\cdots,k$ and all multi-indices $\rho\in\NN^k$ s.t. $|\rho|_\ast=k+1$, and $\sum^\dag$ denotes the summation over $s=1,\cdots,l$ and all $\rho\in\NN^l$ s.t. $|\rho|_\ast=l$. Then $\lv\sum_{j+l=k+1}w_{j,l}(x)\rv\lq C_{q,k}\lambda_1^{-5(k+1)(k+2)}\lambda_q^{k\lb\frac{5}{2}k+6\rb}\kappa_{r+2}^{(k+1)^2}\lb1+|x|^{3(k+1)}\rb$, which is no more than \eqref{induction}, and hence by \eqref{formal} $\tilde{Q}_{k+1}$ has the same bound as \eqref{induction}.

For each $\alpha\in\NN^q$, it is known that the Hermite-type polynomial
\begin{equation*}
H_{\alpha,\Sigma}(x)=\frac{1}{\sqrt{\alpha!}}\prod_{j=1}^qH_{\alpha_j}(\lambda_j^{-1/2}x_j)
\end{equation*}
is the eigenfunction of the differential operator of the equation \eqref{PDE} corresponding to the eigenvalue $\nu_\alpha:=\sum_{j=1}^q\alpha_j\lambda_j^{-1}\lq|\alpha|/\lambda_q$. Since $\{H_{\alpha,\Sigma}\}$ form an orthonormal basis for the Hilbert space $L^2(\RR^q,\phi_\Sigma\td\Lambda)$, the polynomial $u_{k+1}$ can be expressed as
\begin{equation*}
u_{k+1}(x)=\sum_{|\alpha|\lq3(k+1)}c_\alpha\nu_\alpha^{-1}H_{\alpha,\Sigma}(x),
\end{equation*}
where $c_\alpha=\int_{\RR^q}(Q_{k+1}(z)-\tilde{Q}_{k+1}(z))H_{\alpha,\Sigma}(z)\phi_\Sigma(z)\td z$. Then by the Cauchy-Schwartz inequality and \eqref{bound_2}, the above estimate for $\tilde{Q}_{k+1}$ implies that
\begin{align*}
|u_{k+1}(x)|\lq&C_{q,k}\sum_{|\alpha|\lq3(k+1)}C_\alpha\lb\int_{\RR^q}|Q_{k+1}(z)-\tilde{Q}_{k+1}(z)|^2\phi_\Sigma(z)\td z\rb^\frac{1}{2}\lambda_q\lambda_1^{-\frac{|\alpha|}{2}}\lb1+|x|^{|\alpha|}\rb\\
\lq&C_{q,k}\lambda_1^{-(k+1)\lb5k+11\rb-\frac{3}{2}(k+1)}\lambda_q^{(k+1)\lb\frac{5}{2}k+6\rb+\frac{3}{2}(k+1)+1}\kappa_{r+2}^{(k+1)^2}\lb1+|x|^{3(k+1)}\rb\\
\lq&C_{q,k}\lambda_1^{-5(k+1)(k+3)}\lambda_q^{1+\frac{5}{2}(k+1)(k+3)}\kappa_{r+2}^{(k+1)^2}\lb1+|x|^{3(k+1)}\rb,
\end{align*}
which agrees with the induction hypothesis; the initial step for $u_1$ also holds true as $\tilde{Q}_1\equiv0$. Therefore the bound in (i) holds true for each $u_k$, and it holds true for its gradient $p_k$, too. The induction step also gives the bound \eqref{induction} for $\tilde{Q}_{r+1}$, and hence (ii) follows from the triangle inequality and \eqref{bound_2} again.

For a general positive-definite $\Sigma$, one diagonalises it with an orthogonal matrix $A$ and applies the same arguments above to the scaled polynomials $p_k^\ast(x):=A^\top p_k(Ax)$. By Remark \ref{rotation} the $p_k^\ast$'s still lie in $\mathscr{P}_G$, and the results still hold.
\end{proof}

The proof above takes a compromise approach by introducing $\tilde{Q}_{r+1}$ in \eqref{density_approx}: the condition ``$\eps$ sufficiently small" is not needed for Lemma \ref{bijection}, as Proposition 1 in \cite{davie2015Polpernordis} allows an $O(\eps^{r+1})$ estimate for $\int_{\RR^q}|x|^M|\zeta_{\eps,r}(x)-\phi_\Sigma(x)\Q_{\eps,r}(x)|\td x$ for all $\eps>0$. However, whilst practically $K_r=\lambda_1^{-5r(r+2)}\lambda_q^{1+5r(r+2)/2}\kappa_{r+2}^{r^2}$ by (i), it is rather complicated to compute $N_r$ explicitly.

Before proceeding to the main result, given fixed parameters $\beta\in(0,1/3)$ and $\bar{\gamma},\tau\in(0,1)$, it would be convenient to combine all the criteria for $\eps$ together: for any integer $r\gq3$ the statement ``\textbf{$m$ sufficiently large w.r.t. $r$}" or ``\textbf{$\eps$ sufficiently small w.r.t. $r$}" refers to that $m>\kappa_{r+\tau}^{\max\{4,6/(r(1-3\beta))\}}\vee K_{r-3}^{2N_{r-3}}$ with $K_0,N_0:=1$ and that \eqref{large_m} holds for $n=r$.

\subsection{Main Result and Some Special Cases}

Given Lemma \ref{bijection}, it will be shown in the following theorem that the normal distribution $\N(0,\Sigma)$ perturbed by the polynomials $\{p_k\}$ is close to the law $\prb_m$ in the Vaserstein distances. The proof is a more detailed and quantitative version of what is exhibited in Section 4 in \cite{davie2015Polpernordis}, and specifies the dependence on $\Sigma$ and certain moments of $X$.

\begin{theorem}\label{CLT}
Suppose $\Sigma$ is non-singular and $\chi$ satisfies \ref{cramer_interp}. Fix an integer $n\gq3$, an even integer $p\in2\ZZ^+$ and $\beta\in(0,1/3)$. If $\ex|X|^{p(n-2)+2+\tau}<\infty$ for some $\tau\in(0,1)$, then for $m$ sufficiently large w.r.t. $p(n-2)+2$,
\begin{equation*}
\wass_p(Y_m,~\mathbf{p}_{m,n-3}(\xi_\Sigma))\lq C_{p,q,n,\tau}\Xi_Xm^{-(n-2)/2},
\end{equation*}
where $\mathbf{p}_{m,n-3}$ is the polynomial defined by Lemma \ref{bijection}, and $\Xi_X$ is a constant depending on $p,n,\beta,\eta,\Sigma,~\ex|X|^{p(n-2)+1}$ and $\ex|X|^{p(n-2)+2+\tau}$.
\end{theorem}
\begin{proof}
Denote $r=p(n-2)+2$. Taking the inverse Fourier transform, Proposition \ref{prop} implies that for all $x\in\RR^q$ and for $m$ sufficiently large w.r.t. $r$,
\begin{align}\label{3}
|F_{r-2}(\eps,x)|:=&\lv\tilde{f}_m(x)-\phi_\Sigma(x)\Q_{\eps,r-2}(x)\rv\lq C_q\int_{\RR^q}\lv\tilde{\psi}_m(z)-e^{-\frac{1}{2}z\cdot\Sigma z}\P_{\eps,r-2}(z)\rv\td z\nonumber\\
\lq&C_{q,r,\tau}(\det\Sigma)^{-\frac{1}{2}}\lambda_1^{-\frac{r-1}{2\beta}}\kappa_{r+\tau}^{r-2}\eps^{r-1}.
\end{align}
The goal is to use the inequality \eqref{crude_bound} to bound the $\wass_p$ distance, for which one first writes
\begin{equation*}
\int_{\RR^q}|x|^p|F_{r-2}(\eps,x)|\td x\lq\int_{\RR^q}|x|^p\lb\tilde{f}_m(x)+\phi_\Sigma(x)\lv\Q_{\eps,r-2}(x)\rv\rb\td x\lq I_1+2I_2+I_3,
\end{equation*}
where, for any $\eta\in(0,1)$,
\begin{align*}
I_1:=&\int_{|x|\lq\eps^{-\eta/(p+q)}}|x|^p|F_{r-2}(\eps,x)|\td x,~I_2:=\int_{|x|>\eps^{-\eta/(p+q)}}|x|^p\phi_\Sigma(x)\lv\Q_{\eps,r-2}(x)\rv\td x,\\
I_3:=&\int_{|x|>\eps^{-\eta/(p+q)}}|x|^pF_{r-2}(\eps,x)\td x.
\end{align*}
For any fixed $p\gq2$ and $\eta\in(0,1)$, one finds, by virtue of \eqref{3},
\begin{equation*}
I_1\lq C_{q,r,\tau}(\det\Sigma)^{-\frac{1}{2}}\lambda_1^{-\frac{r-1}{2\beta}}\kappa_{r+\tau}^{r-2}\eps^{r-1-\eta}.
\end{equation*}
For the integral on the complement $\{x:~|x|>\eps^{-\eta/(p+q)}\}=\{1<\eps|x|^{(p+q)/\eta}\}$, one gets
\begin{align}
I_2\lq&\int_{|x|>\eps^{-\eta/(p+q)}}|x|^p\phi_\Sigma(x)\kappa_r^{r-2}\lb1+\sum_{k=1}^{r-2}\eps^k|x|^{3k}\rb\td x\nonumber\\
\lq&C_r\int_{|x|>\eps^{-\eta/(p+q)}}|x|^p\phi_\Sigma(x)\kappa_r^{r-2}\eps^{r-1}|x|^{\frac{p+q}{\eta}(r-1)}\td x\lq C_r\lambda_q^{\frac{p}{2}+\frac{p+q}{2\eta}(r-1)}\kappa_r^{r-2}\eps^{r-1},
\end{align}
due to the fact that $(p+q)/\eta>3$ and \eqref{bound_2}. Also observe that
\begin{equation*}
I_3\lq\int_{\RR^q}|x|^p\lb\tilde{f}_m(x)-\phi_\Sigma(x)\Q_{\eps,r-2}(x)\rb\td x+I_1=:I_4+I_1,
\end{equation*}
by the triangle inequality. In order to get a good estimate for $I_4$, first observe that $\forall p\gq2$ by Rosenthal's inequality - see e.g. Lemma 1 in \cite{RatAppMulInvPriSumiidRanVecFinMom},
\begin{equation}\label{rosenthal}
\int_{\RR^q}|x|^p\prb_m(\td x)=\ex|Y_m|^p=m^{-\frac{p}{2}}\ex\lv\sum_{j=1}^mX_j\rv^p\lq C_p\lb m^{1-\frac{p}{2}}\ex|X|^p+(\ex|X|^2)^\frac{p}{2}\rb.
\end{equation}
Also, from the construction of $\theta_\eps$ \eqref{theta_def} (now supported on $\{|x|<\eps^{r+1}\}$),
\begin{equation}\label{theta}
\int_{\RR^q}|x|^p\theta_\eps(\td x)=\int_{|y|<1}|y|^p\eps^{p(r+1)}h(y)\td y<\eps^{p(r+1)},
\end{equation}
by a change of variables. For an even $p\gq4$, as $p<r$ observe that all the moments up to $p$ of the expansion $\phi_\Sigma\Q_{\eps,r-2}$ match those of $Y_m$ by Remark \ref{moment_match}. Hence by \eqref{rosenthal} and \eqref{theta},
\begin{align*}
I_4\lq&\int_{\RR^q}\int_{\RR^q}\lb|x+y|^p-|x|^p\rb\prb_m(\td x)\theta_\eps(\td y)\lq C_{p,q}\sum_{k=1}^p\int_{\RR^q}\int_{\RR^q}|x|^{p-k}|y|^k\prb_m(\td x)\theta_\eps(\td y)\\
\lq&C_{p,q}\eps^{r+1}\int_{\RR^q}|x|^{p-1}\prb_m(\td x)\lq C_{p,q}\eps^{r+1}\lb\eps^{p-3}\ex|X|^{p-1}+\lambda_q^\frac{p-1}{2}\rb;
\end{align*}
for $p=2$ the bound is reduced to $C_{p,q}\eps^{r+1}(\ex|X|^2)^{1/2}$ by \eqref{rosenthal} and H\"older's inequality. Therefore, altogether one arrives at, for $p\lq r$,
\begin{equation*}
\int_{\RR^q}|x|^p|F_{r-2}(\eps,x)|\td x\lq C_{p,q,r,\tau}\lb(\det\Sigma)^{-\frac{1}{2}}\lambda_1^{-\frac{r-1}{2\beta}}+\lambda_q^{\frac{p}{2}+\frac{p+q}{2\eta}(r-1)}\rb\kappa_{r+\tau}^{r-2}\eps^{r-1-\eta}.
\end{equation*}
Finally by the triangle inequality one removes the $(r-2)$-th term in $\Q_{\eps,r-2}$:
\begin{equation*}
\int_{\RR^q}|x|^p|F_{r-2}(\eps,x)|\td x\gq\int_{\RR^q}|x|^p|F_{r-3}(\eps,x)|\td x-C_{q,r}\eps^{r-2}\kappa_r^{r-2}\int_{\RR^q}|x|^p\lb|x|^r+|x|^{3(r-2)}\rb\phi_\Sigma\td x,
\end{equation*}
and uses \eqref{bound_2} again to deduce the following estimate for the Edgeworth approximation:
\begin{equation*}
\int_{\RR^q}|x|^p|F_{r-3}(\eps,x)|\td x\lq C_{p,q,r}\lb(\det\Sigma)^{-\frac{1}{2}}\lambda_1^{-\frac{r-1}{2\beta}}+\lambda_q^{\frac{p}{2}+\frac{p+q}{2\eta}(r-1)}\rb\kappa_{r+\tau}^{r-2}\eps^{r-2}.
\end{equation*}

Since the smooth measure $\theta_\eps$ is also supported on $\{x:|x|<\eps^{r-2}\}$, the estimate above implies that the Edgeworth polynomials $\{Q_k\}\in\mathscr{P}_\Sigma$ form an $\mathcal{A}_\Sigma$-sequence for the family of probability measures $\{\prb_m\}$ - see Definition 1 in \cite{davie2015Polpernordis}. Then one can extend the expansion $\Q_{\eps,r-3}$ to a larger value of $r$ and take the $p$-th root to get a $\wass_p$ estimate, as in done in the proof of Theorem 4 in \cite{davie2015Polpernordis}.

With the polynomials $\{p_k\}=\mathfrak{S}_\Sigma^{-1}(\{Q_k\})$ and $\mathbf{p}_{\eps,r-3},~R_{\eps,r-3}$ defined as in Lemma \ref{bijection}, using the triangle inequality and the inequality \eqref{crude_bound}, one can deduce the following estimate for an integer $n\gq3$ by replacing $r=p(n-2)+2$ in the estimate:
\begin{align*}
%&\wass_p\lb Y_m,~\xi_\Sigma+\sum_{k=1}^{p(n-2)-1}\eps^kp_k(\xi_\Sigma)\rb\\
&\wass_p\lb Y_m,~\mathbf{p}_{\eps,p(n-2)-1}(\xi_\Sigma)\rb\lq C_p\lb\int_{\RR^q}|x|^p|F_{p(n-2)-1}|\td x+\int_{\RR^q}|x|^p\lv R_{\eps,p(n-2)-1}(x)\rv\td x\rb^\frac{1}{p}\\
%&\lq C_{p,q,n}\left\{|\Sigma|^{-\frac{1}{2}}\lambda_1^{-\frac{3}{2}(p(n-2)+1)}\Gamma_{p(n-2)+2,\beta}+\lambda_q^\frac{p}{2}+\lambda_q^{\frac{p}{2}+\frac{p+q}{2\eta}(p(n-2)+1)}+\ex|X|^{p-1}\right\}^\frac{1}{p}\eps^{n-2}\\
&\lq C_{p,q,n,\tau}\lb(\det\Sigma)^{-\frac{1}{2p}}\lambda_1^{-\frac{1}{2\beta}\lb n-2+\frac{1}{p}\rb}+\lambda_q^{\frac{1}{2}+\frac{p+q}{2\eta}\lb n-2+\frac{1}{p}\rb}\rb\kappa_{p(n-2)+2+\tau}^{n-2}\eps^{n-2}\\
&\quad+C_{p,q,n}\lambda_1^{-5(n-2)(p(n-2)+1)}\lambda_q^{\frac{1}{2}+\frac{5}{2}(n-2)(p(n-2)+1)}\kappa_{p(n-2)+1}^{p(n-2)^2}\eps^{n-2},
\end{align*}
whilst the excess terms from $n-2$ to $p(n-2)-1$ can be handled by (i) of Lemma \ref{bijection} and \eqref{bound_2} again:
\begin{align*}
%&\wass_p\lb\xi_\Sigma+\sum_{k=1}^{p(n-2)-1}\eps^kp_k(\xi_\Sigma),~\xi_\Sigma+\sum_{k=1}^{n-3}\eps^kp_k(\xi_\Sigma)\rb\\
\wass_p\lb\mathbf{p}_{\eps,p(n-2)-1}(\xi_\Sigma),~\mathbf{p}_{\eps,n-3}(\xi_\Sigma)\rb\lq&C_p\lb\ex\lv\sum_{k=n-2}^{p(n-2)-1}\eps^kp_k(\xi_\Sigma)\rv^p\rb^\frac{1}{p}\\
\lq&C_{p,q,n}\lambda_1^{-5p^2(n-2)^2+5}\lambda_q^{\frac{5}{2}p^2(n-2)^2+\frac{3}{2}p(n-2)-3}\kappa_{p(n-2)+1}^{(p(n-2)-1)^2}\eps^{n-2}.
\end{align*}
Thus the claimed result follows from the triangle inequality.
\end{proof}
\begin{remark}
The number of moments of $X$ needed for Theorem \ref{CLT} is independent of the dimension $q$.
\end{remark}

The optimal result for the central limit theorem for $q=1$ is already given in \cite{BerBouEdgExpCenLimTheTraDis}, which is not fully recovered by Theorem \ref{CLT} as the inequality \eqref{crude_bound} is rather crude compared to \eqref{wass_1d}. For $q\gq2$, Theorem \ref{CLT} immediately implies the following (by choosing $n=3$):
\begin{corollary}\label{n=3}
Suppose the i.i.d. random variables $\{X_j\}$ have non-singular covariance $\Sigma$ and satisfy \ref{cramer_interp}, and let $p\in2\ZZ^+$. If $\ex|X|^{p+2+\tau}<\infty$ for some $\tau\in(0,1)$, then by taking e.g. $\beta=1/6,~\eta=1/2$, the following holds for $m$ sufficiently large w.r.t. $p+2$:
\begin{align*}
\wass_p(Y_m,~\xi_\Sigma)\lq&C_{p,q,\tau}\lb(\det\Sigma)^{-1/(2p)}\lambda_1^{-3(1+1/p)}+\lambda_q^{(p+q)(1+1/p)+1/2}\rb\kappa_{p+2+\tau}m^{-1/2}\\
&+C_{p,q}\lambda_1^{-5p^2+5}\lambda_q^{(5p^2+3p)/2-3}\kappa_{p+1}^{2\vee(p-1)^2}m^{-1/2}.
\end{align*}
\end{corollary}
As mentioned in the introduction, for the special case $p=2$ this corollary is weaker than the results of Bobkov \cite{bobkov2013RatConEdgExpEntCenLimThe} and Bonis \cite{bonis2016RatCenLimTheDifAppviaSteMet}. Although the condition $\ex|X|^{4+\tau}<\infty$ is slightly better than that of Bobkov, he does not require Cram\'er's condition as he aimed at estimating the relative entropy $\mathds{H}(\prb_m\|\N(0,\Sigma))$, which in turn gives a bound for the $\wass_2$ distance according to Talagrand's transport inequality \cite{talagrand1996TraCosGauOthProMea}. On the other hand, Bonis' optimal result relies on a differential estimate for the $\wass_2$ distance in terms of the Fisher information. As the $\wass_2$ distance enjoys those special properties, the inequality \eqref{crude_bound} does not provide a sharp bound.

However, Theorem \ref{CLT} can potentially give higher-order convergence if one considers a non-trivial expansion ($n>3$). For example, when choosing $n=4$, one gets a rate $O(m^{-1})$ under Cram\'er's condition and that $\ex|X|^{6+\tau}<\infty$. The task is to find the polynomial $p_1$ using the method described in the discussion before Lemma \ref{bijection}: given $Q_1(x)$ defined in \eqref{hermite}, one looks for the unique (up to an additive constant) polynomial solution $u_1:\RR^q\to\RR$ satisfying \eqref{PDE} for the initial step:
\begin{equation*}
-\Delta u_1(x)+x\cdot\Sigma^{-1}\nabla u_1(x)=Q_1(x).
\end{equation*}
To illustrate that consider the simplest case where $q=2$ and $\Sigma=I$. The cubic polynomial $6iP_1(z)=\mu_{(3,0)}z_1^3+3\mu_{(2,1)}z_1^2z_2+3\mu_{(1,2)}z_1z_2^2+\mu_{(0,3)}z_2^3$ gives
\begin{equation*}
6Q_1(x)=\mu_{(3,0)}H_3(x_1)+3\mu_{(2,1)}H_2(x_1)H_1(x_2)+3\mu_{(1,2)}H_1(x_1)H_2(x_2)+\mu_{(0,3)}H_3(x_2).
\end{equation*}
Notice that $x\cdot\nabla u(x)=ku(x)$ for any monomial $u$ of degree $k$, and so the polynomial solution to the PDE above is cubic with no quadratic terms. Then using the property $H_j'=jH_{j-1}$ and matching the coefficients on both sides of the equation, one gets
\begin{align*}
u_1(x)=&\frac{1}{18}\mu_{(3,0)}H_3(x_1)+\frac{1}{6}\mu_{(2,1)}H_2(x_1)H_1(x_2)+\frac{1}{6}\mu_{(1,2)}H_1(x_1)H_2(x_2)+\frac{1}{18}\mu_{(0,3)}H_3(x_2)\\
&+\frac{1}{3}(\mu_{(3,0)}+\mu_{(1,2)})H_1(x_1)+\frac{1}{3}(\mu_{(0,3)}+\mu_{(2,1)})H_1(x_2)+C,
\end{align*}
and hence the perturbing polynomial $p_1=\nabla u_1$ is found.

Under certain stronger conditions, one can also obtain higher-order convergence without specifying the perturbing polynomials $p_k$. For $q=1$, Bobkov \cite{BerBouEdgExpCenLimTheTraDis} (Theorem 1.3) proved that if $\ex X^k=\ex\xi_\Sigma^k$ for $k=1,\cdots,n-1,~3\lq n\in\ZZ^+$, and $\ex|X|^{p(n-2)+2}<\infty$, then under Cram\'er's condition one has $\wass_p(Y_m,\xi_\Sigma)=O(m^{-(n-2)/2})$. This can be readily generalised to $q\gq2$ by Theorem \ref{CLT}: if the first $n-1$ moments match those of $\N(0,\Sigma)$, the cumulants $\mu_\alpha(X)=\mu_\alpha(\xi_\Sigma)=0$ for all $|\alpha|=3,\cdots,n-1$, implying that $P_k=Q_k\equiv0$. This immediately implies that $\mathcal{L}_\Sigma p_k\equiv0$ in \eqref{PDE} for all $k=1,\cdots,n-3$, forcing $p_k\equiv0$. Therefore one asserts the following:
\begin{corollary}
Suppose the i.i.d. random variables $\{X_j\}$ with non-singular covariance $\Sigma$ satisfy Cram\'er's condition and let $p\in2\ZZ^+$. If $\exists3\lq n\in\ZZ^+$ s.t. $\ex X^\alpha=\ex\xi_\Sigma^\alpha$ for all $|\alpha|=1,\cdots,n-1$, and $\ex|X|^{p(n-2)+2+\tau}<\infty$ for some $\tau\in(0,1)$, then $\wass_p(Y_m,\xi_\Sigma)=O(m^{-(n-2)/2})$ for $m$ sufficiently large w.r.t. $p(n-2)+2$.
\end{corollary}

\section{Application to the Euler Approximation of {\levy} SDEs}

Consider the $d$-dimensional SDE \eqref{sde} driven by a $q$-dimensional {\levy} process \eqref{noise}. Assume that the {\levy} measure $\nu$ has finite second moment, and the function $\sigma:\RR^d\to\RR^{d\times q}$ is bounded and Lipschitz. It will be shown in this section that the $q$-dimensional small jumps \eqref{small_jumps} can also be approximated by a normal random variable with rate $1$ while the computational cost $E_\nu(h)$ remains controlled for $\nu$ satisfying certain stable-like conditions, in particular \ref{alpha-stable}.

\subsection{Normal Approximation of Small Jumps via Central Limit Bound}

The way both Fournier \cite{fournier2011SimappLevstodifequ} and Godinho \cite{AsyGraColSpaHomBolEquSofCouPot} applied the central limit theorem for the small jumps $Z_t^\epsilon$ is to split the time interval into $m$ subdivisions and view $Z_t^\epsilon$ as the sum of the i.i.d. random variables $\int_{(j-1)t/m}^{jt/m}\int_{0<|z|\lq\epsilon}z\tilde{N}(\td z,\td s),~j=1,\cdots,m$. Here an alternative approach is considered: one may decompose the range of the jumps $\{0<|z|\lq\epsilon\}$ into countably many annuli and represent the small jumps as a sum:
\begin{equation}\label{levy}
Z_t^\epsilon=\sum_{r=r_0}^\infty\int_0^t\int_{\Omega_r}z\tilde{N}(\td z,\td s)=:\sum_{r=r_0}^\infty V_t^r,
\end{equation}
where $\Omega_r=\{2^{-r-1}<|z|\lq2^{-r}\}$ and $r_0=-\log_2\epsilon>0$. Assume $\nu$ to be $\sigma$-finite and denote $\nu_r:=\nu(\Omega_r)$. By the {\levy-\ito} decomposition one knows that each $V_t^r$ is a compensated compound Poisson process:
\begin{equation}\label{jumps}
V_t^r=\sum_{j=1}^{N_t^r}X_j^r-t\nu_r\ex X_j^r,
\end{equation}
where $\{X_j^r\}$ are i.i.d. random variables bounded within $\Omega_r$ and $N_t^r$ follows $\poi(t\nu_r)$.

Instead of directly working with $V_t^r$, one may first consider a general compound Poisson process $V_t$ of the form \eqref{jumps} with $N_t\sim\poi(t\mu)$ and the jumps $X_j$ on the interval $[0,1]$. Expecting $\mu$ to be large, one can write
\begin{equation*}
Y=\mu^{-\frac{1}{2}}V_1=\mu^{-\frac{1}{2}}\sum_{j=1}^{N_1}X_j-\mu^\frac{1}{2}\ex X_j,
\end{equation*}
and approximate it by Edgeworth-type polynomials using the same recipe just as before.

Let $\psi$ and $\chi$ be the characteristic functions of $Y$ and the $X_j$'s, respectively, and $\Sigma_X$ be the covariance of $X$, with eigenvalues $\lambda_{1,X}\lq\cdots\lq\lambda_{q,X}$, and similarly $\kappa_{M,X}=1\vee\ex|X|^M,~\forall M>0$. Then one has the following simple relation between the distributions of $X$ and $Y$:
\begin{equation*}
\psi(z)=\exp\left\{\mu\lb\chi\lb\mu^{-\frac{1}{2}}z\rb-1\rb-i\mu^\frac{1}{2}z\cdot\ex X\right\}.
\end{equation*}
Given this convenient expression, instead of taking the logarithm one may directly apply Taylor expansion to $\chi$ and have, instead of \eqref{psi_expand}, a formal expansion
\begin{equation*}
\psi(z)\sim e^{-\frac{1}{2}z\cdot\Sigma_Xz}\lb1+\sum_{k=1}^\infty\mu^{-\frac{k}{2}}P_k(z)\rb,
\end{equation*}
whose $(n-2)$-th truncation leads to the same bound as in Lemma \ref{char_estimate_1} with $\mu$ in place of $m$ and $\eps=\mu^{-1/2}$ for $|z|\lq\mu^{\beta/2},~\beta\in(0,1/3)$. Note that the $P_k$ here are slightly different (in fact simpler): since no logarithm is taken, the cumulants $\mu_\alpha$ are replaced with just $\ex X^\alpha$. Also for $|z|\lq\mu^{1/2}\delta=\mu^{1/2}\min\{\lambda_{1,X}/\kappa_{3,X},~\kappa_{n,X}^{-1/n}/2\}$, one still has
\begin{equation*}
|\psi(z)|=\lv\exp\left\{\mu\lb\chi\lb\mu^{-\frac{1}{2}}z\rb-1\rb-i\mu^\frac{1}{2}z\cdot\ex X\right\}\rv\lq e^{-\frac{1}{2}z\cdot\Sigma_Xz+\frac{1}{6}\mu^{-\frac{1}{2}}\ex|X|^3|z|^3}\lq e^{-\frac{1}{4}z\cdot\Sigma_Xz}.
\end{equation*}
Moreover, by imposing Cram{\'e}r's condition \ref{cramer_interp} on the distribution of $X$, one can still achieve a similar bound for $|\psi|$:
\begin{align}
|\psi(z)|=&\lv\exp\left\{\chi\lb\mu^{-\frac{1}{2}}z\rb-1-i\mu^{-\frac{1}{2}}z\cdot\ex X\right\}\rv^\mu\nonumber\\
=&\lb\exp\left\{\textrm{Re}\chi\lb\mu^{-\frac{1}{2}}z\rb-1\right\}\rb^\mu\lq\lb e^{\bar{\gamma}-1}\rb^\mu\in(0,1),\label{gamma_bar}
\end{align}
for $|z|>\mu^{1/2}\delta$ and some $\bar{\gamma}\in(0,1)$ according to Lemma \ref{cramer}. Thus one arrives at virtually the same estimate as in Proposition \ref{prop}, and therefore Theorem \ref{CLT} still holds true for $\eps=\mu^{-1/2}$ sufficiently small w.r.t. $p+2$, and Corollary \ref{n=3} applies with $\mu$ in place of $m$ and $\exp(\bar{\gamma}-1)$ in place of $\bar{\gamma}$. For the normal approximation ($n=3$), since no perturbing polynomials $p_k$ are concerned, one can scale the jumps $\widehat{X}:=\Sigma_X^{-1/2}X$ and deduce, $\forall\tau\in(0,1)$,
\begin{equation}\label{compound_PP}
\wass_p\lb V_1,\mu^\frac{1}{2}\xi_{\Sigma_X}\rb\lq\lvv\Sigma_X^\frac{1}{2}\rvv\wass_p\lb\Sigma_X^{-\frac{1}{2}}V_1,\mu^\frac{1}{2}\xi_I\rb\lq C_{p,q,\tau}\kappa_{p+2+\tau,\widehat{X}}^{2\vee(p-1)^2}\lambda_{q,X}^{1/2}.
\end{equation}
One can apply the above arguments to the jump process \eqref{jumps} by scaling the jump sizes. For the jumps $X_j^r$ on each annulus $\Omega_r$, define $X_j:=2^rX_j^r$ and $\widehat{X}_j:=\Sigma_X^{-1/2}X_j$ accordingly. For each fixed $r$, the $X_j^r$'s are i.i.d. with characteristic function
\begin{equation*}
\chi^r(s)=\nu_r^{-1}\int_{\Omega_r}e^{is\cdot x}\nu(\td x).
\end{equation*}
This implies that $X$ has scaled covariance $\Sigma_X=\nu_r^{-1}2^{2r}\int_{\Omega_r}xx^\top\nu(\td x)$ with eigenvalues $\lambda_{j,X}=\nu_r^{-1}2^{2r}\lambda_{j,r}$, where $\lambda_{1,r}\lq\cdots\lq\lambda_{q,r}$ are the eigenvalues of $\Sigma_r:=\int_{\Omega_r}xx^\top\nu(\td x)$. Also notice that $\ex|X|^M=\nu_r^{-1}2^{rM}\int_{\Omega_r}|x|^M\nu(\td x)\lq1,~\forall M>0$, implying that $\ex|\widehat{X}|^M\lq\lambda_{1,X}^{-M/2}$.

Thus, if $\Sigma_r$ is non-singular for each $r$, then (assuming $\lambda_{1,X}\lq1$ w.l.o.g.) one can apply \eqref{compound_PP} with parameter $\mu=t\nu_r$:
\begin{equation}\label{wass_r}
\wass_p\lb V_t^r,\sqrt{t}\xi_{\Sigma_r}\rb=2^{-r}\wass_p\lb2^rV_t^r,\sqrt{t\nu_r}\xi_{\Sigma_X}\rb\lq C_{p,q,\tau}\lambda_{1,X}^{-(1\vee\frac{(p-1)^2}{2})(p+2+\tau)}\lambda_{q,r}^{1/2}\nu_r^{-1/2}.
\end{equation}
Denote further $\Sigma_\epsilon:=\int_{0<|x|\lq\epsilon}xx^\top\nu(\td x)$, then from this bound one can find a coupling between $Z_t^\epsilon$ and $\N(0,t\Sigma_\epsilon)$ under suitable conditions.

\begin{theorem}\label{jump_normal_approx}
Suppose $\xi_r(s):=\chi^r(2^rs)$ satisfies \ref{cramer_interp} uniformly for all $r\gq r_0$. If $\nu_r^{-1}=o(2^{-r})$ as $r\to\infty$, then $\forall p\in2\ZZ^+,~t\gq\epsilon$ and $\epsilon$ sufficiently small,
\begin{equation*}
\wass_p\lb Z_t^\epsilon,\sqrt{t}\xi_{\Sigma_\epsilon}\rb\lq C_{p,q}\epsilon.
\end{equation*}
\end{theorem}
\begin{proof}
Note that on each $\Omega_r$ it is always true that $\lambda_{q,r}\lq\tr\Sigma_r\lq2^{-2r}\nu_r$ and $\lambda_{q,r}\gq q^{-1}\tr\Sigma_r\gq q^{-1}2^{-2(r+1)}\nu_r$. Write $\xi_r(s)=|\xi_r(s)|e^{i\theta}$, where $\theta=\theta(r,s)$ satisfies $\int_{\Omega_r}\sin(2^rs\cdot x-\theta)\nu(\td x)=0$. Then $\int_{\Omega_r}\sin(2^rs\cdot x)\nu(\td x)=\tan\theta\int_{\Omega_r}\cos(2^rs\cdot x)\nu(\td x)$ if $\theta\not\equiv\pm\pi/2$ mod $\pi$, and otherwise $\int_{\Omega_r}\cos(2^rs\cdot x)\nu(\td x)=0$. By the uniform Cram\'er's condition for $\xi_r(s)$, there exist $\rho>0,~\gamma\in(0,1)$ s.t. $\forall r\gq r_0$ and $|s|\gq\rho$,
\begin{equation*}
|\xi_r(s)|=\nu_r^{-1}\int_{\Omega_r}\cos(2^rs\cdot x-\theta)\nu(\td x)\in[0,\gamma].
\end{equation*}
If $\theta\not\equiv\pm\pi/2$ mod $\pi$, expand out the integrand using the identity $\cos(\alpha-\beta)=\cos\alpha\cos\beta+\sin\alpha\sin\beta,~\forall\alpha,\beta\in\RR$, replace the term $\int_{\Omega_r}\sin(2^rs\cdot x)\nu(\td x)$ and rearrange to get
\begin{equation*}
(\nu_r\cos\theta)^{-1}\int_{\Omega_r}\cos(2^rs\cdot x)\nu(\td x)\in[0,\gamma].
\end{equation*}
Therefore, regardless of the values of $\theta$, one always has $\lv\int_{\Omega_r}\cos(2^rs\cdot x)\nu(\td x)\rv\lq\gamma\nu_r$ for $|s|\gq\rho$. Write $s=|s|v$ where $v\in\mathbb{S}^{q-1}$ is a unit vector. Then for $|s|\gq\rho$,
\begin{align*}
v\cdot\Sigma_rv=&\int_{\Omega_r}|v\cdot x|^2\nu(\td x)\gq2^{-2r+2}|s|^{-2}\int_{\Omega_r}\sin^2(2^{r-1}s\cdot x)\nu(\td x)\\
&=2^{-2r+1}|s|^{-2}\int_{\Omega_r}\lb1-\cos(2^rs\cdot x)\rb\nu(\td x)\gq2^{-2r+1}\rho^{-2}(1-\gamma)\nu_r.
\end{align*}
This means $\lambda_{1,r}\gtrsim2^{-2r}\nu_r$ by choosing $v$ to be the eigenvector of $\lambda_{1,r}$. Hence $\lambda_{1,r}\simeq\lambda_{q,r}\simeq2^{-2r}\nu_r$ and $\lambda_{1,X}=\nu_r^{-1}2^{2r}\lambda_{1,r}\simeq1,~\forall r\gq r_0$.

Since $\xi_r(s)$ is the characteristic function of $X=2^rX^r$, the uniform Cram\'er's condition validates the bound \eqref{gamma_bar} with a uniform $\bar{\gamma}=\bar{\gamma}(\rho,\gamma)$ and \eqref{compound_PP} holds with $\mu=t\nu_r\gq\epsilon\nu_r\gq2^{-r}\nu_r$ sufficiently large w.r.t. $p+2$. More precisely, one can choose $\epsilon$ sufficiently small s.t. for all $r\gq r_0$, similar to \eqref{large_m},
\begin{equation*}
\lb e^{\bar{\gamma}-1}\rb^{2^{-r}\nu_r}\lb2^{-r}\nu_r\rb^{(q+1)(p+3)/2}\lesssim1.
\end{equation*}
Thus, all the arguments leading towards \eqref{wass_r} are justified, which is immediately reduced to the bound $\wass_p\lb V_t^r,\sqrt{t}\xi_{\Sigma_r}\rb\lq C_{p,q}2^{-r}$, and therefore
\begin{equation*}
\wass_p\lb Z_t^\epsilon,\sqrt{t}\xi_{\Sigma_\epsilon}\rb=\wass_p\lb\sum_{r=r_0}^\infty V_t^r,\sum_{r=r_0}^\infty\sqrt{t}\xi_{\Sigma_r}\rb\lq C_{p,q}\sum_{r=r_0}^\infty2^{-r}=C_{p,q}\epsilon.
\end{equation*}
\end{proof}

Together with the finite second moment of $\nu$, the theorem above requires the order of $\nu_r$ is between $O(2^{r+})$ and $O(2^{2r})$ as $r\to\infty$, i.e. the L\'evy measure needs to be sufficiently singular near $0$. The uniform Cram\'er condition requires certain comparability between $\nu$ and the Lebesgue measure $\Lambda$, conditional on $\Omega_r$. The following lemma gives a sufficient condition.
%\begin{lemma}\label{sin_integral_r}
%If $\nu$ satisfies \ref{alpha-stable} and $\epsilon<\tau$, then there is a constant $\gamma'=\gamma'(q)\in(0,1/2)$ s.t. $\forall|s|\gq\pi2^r,~\theta\in[-\pi,\pi]$ and $\forall r\gq r_0$,
%\begin{equation*}
%A_{\Omega_r}(s,\theta):=\int_{\Omega_r}2\sin^2\frac{1}{2}(s\cdot x+\theta)\nu(\td x)\gq2^{-q-\alpha}\gamma'\nu_r.
%\end{equation*}
%\end{lemma}
%\begin{proof}
%Consider the following sets on each of which $\sin^2(s\cdot x/2+\theta/2)\gq1/2$:
%\begin{equation*}
%D_k=D_k(s,\theta):=\left\{x\in\Omega_r:~2k\pi+\pi/2\lq s\cdot x+\theta\lq2k\pi+3\pi/2\right\},~k\in\ZZ.
%\end{equation*}
%Each $D_k$ is a ``stripe" across the annulus $\Omega_r$ with width $\pi/|s|$ and equally distanced from each other by $\pi/|s|$. This can be seen by rotating so that $s$ lies on one axis. Thus for $|s|\gq\pi2^r$ there is at least one non-empty $D_k$. With $\Lambda$ being the Lebesgue measure on $\RR^q$, the ratio $\Lambda\lb\bigcup_kD_k\rb/\Lambda(\Omega_r)$ approaches $1/2$ by symmetry as $|s|\to\infty$, regardless of the translation $\theta$. Thus the ratio is bounded from below by some $\gamma'=\gamma'(q)\in(0,1/2)$ for all $|s|\gq\pi2^r$, and 
%\begin{equation*}
%\frac{\nu\lb\bigcup_kD_k\rb}{\nu(\Omega_r)}\gq\frac{2^{-(r+1)(q+\alpha)}\Lambda\lb\bigcup_kD_k\rb}{2^{-r(q+\alpha)}\Lambda(\Omega_r)}\gq2^{-q-\alpha}\gamma',
%\end{equation*}
%and the result follows from that $A_{\Omega_r}(s,\theta)\gq A_{\bigcup_kD_k}(s,\theta)\gq\nu\lb\bigcup_kD_k\rb$.
%\end{proof}
\begin{lemma}\label{sufficient}
If there exist $a,b\in(0,1)$ s.t. for each $r\gq r_0$, any measurable subset $\Gamma_r$ of $\Omega_r$ satisfying $\Lambda(\Gamma_r)/\Lambda(\Omega_r)\gq a$ must have that $\nu(\Gamma_r)/\nu(\Omega_r)\gq b$, then $\xi_r(s)=\chi^r(2^rs)$ satisfies \ref{cramer_interp} uniformly for all $r\gq r_0$.
\end{lemma}
\begin{proof}
For any $a'\in(0,1)$ denote $b'=\sin^2\frac{\pi}{2}(1-a')\in(0,1)$. For any $\theta\in\RR,~v\in\RR^q$, consider
\begin{equation*}
D_k=D_k(v,\theta):=\left\{x\in\Omega_r:~2k\pi+(1-a')\pi\lq v\cdot x-\theta\lq2(k+1)\pi-(1-a')\pi\right\},~k\in\ZZ,
\end{equation*}
on each of which $\sin^2\frac{1}{2}(v\cdot x-\theta)\gq b'$. They are parallel ``stripes" across the annulus $\Omega_r$ with width $2a'\pi/|v|$ equidistantly away from each other by $2(1-a')\pi/|v|$. This can be seen by rotating so that $v$ lies on one axis. Thus for $|v|>\pi2^{r+1}$ there is at least one non-empty $D_k$. Denote $\Gamma_r=\bigcup_{k\in\ZZ}D_k$, then the ratio $\Lambda(\Gamma_r)/\Lambda(\Omega_r)$ approaches $a'$ as $|v|\to\infty$, regardless of the translation $\theta$. Therefore one can find some constants $\rho>0$ and $\gamma'=\gamma'(\rho,q)\in(0,a')$ s.t. for all $|v|\gq2^r\rho,~\Lambda(\Gamma_r)/\Lambda(\Omega_r)\gq\gamma'$. Choose $\gamma'=a$ as given in the assumption, then $\nu(\Gamma_r)/\nu(\Omega_r)\gq b$ for all $|v|\gq2^r\rho$. Write $\xi_r(s)=|\xi_r(s)|e^{i\theta}$ for some $\theta=\theta(r,s)$, then
\begin{equation*}
1-|\xi_r(s)|\gq2\nu_r^{-1}\int_{\Gamma_r}\sin^2\frac{1}{2}(2^rs\cdot x-\theta)\nu(\td x)\gq2b'\nu(\Gamma_r)/\nu(\Omega_r)\gq2b'b,
\end{equation*}
for all $r\gq r_0$, and the result follows by setting $v=2^rs$ and $\gamma=1-2b'b\in(0,1)$.
\end{proof}

\begin{corollary}\label{clt_approach}
If $\nu$ satisfies \ref{alpha-stable} for $\alpha\in(1,2)$, then Theorem \ref{jump_normal_approx} holds for $\epsilon\in(0,\tau)$ sufficiently small.
\end{corollary}
\begin{proof}
One just needs to check that the assumptions in Theorem \ref{jump_normal_approx} are satisfied. First of all, for $\alpha\in(1,2)$ and $\forall r\gq r_0$,
\begin{equation*}
\nu_r\simeq\int_{\Omega_r}|x|^{-q-\alpha}\td x=C_q\frac{2^\alpha-1}{\alpha}2^{\alpha r}.
\end{equation*}
Then $2^r\nu_r^{-1}=o(1)$ for $\alpha\in(1,2)$. For any measurable subset $\Gamma_r$ of $\Omega_r$,
\begin{equation*}
\frac{\nu(\Gamma_r)}{\nu(\Omega_r)}\gq\frac{\int_{\Gamma_r}|x|^{-q-\alpha}\td x}{\int_{\Omega_r}|x|^{-q-\alpha}\td x}\gq2^{-q-\alpha}\frac{\Lambda(\Gamma_r)}{\Lambda(\Omega_r)},
\end{equation*}
which validates Lemma \ref{sufficient}.
\end{proof}

It is worth mentioning that if the condition \ref{alpha-stable} is assumed a priori, then one could directly use the L\'evy-Khintchine formula to study the global behaviour of the characteristic function of $Z_t^\epsilon$, which would greatly simplify the analysis of Section 1, but the same arguments used in the proof of Theorem \ref{CLT} would still be needed for the coupling result.

\subsection{A Coupling for Euler's Approximation}

Given the coupling result above, one finally arrives at the stage of recovering Fournier's results \cite{fournier2011SimappLevstodifequ} on the Euler approximation of the SDE \eqref{sde} driven by the L\'evy process \eqref{noise}:
\begin{align}
x_t=&x_0+\int_0^t\sigma(x_{s-})\td Z_s,~t\in[0,T],\label{sde_2}\\
Z_t=&at+BW_t+\int_0^t\int_{\RR^q\setminus\{0\}}z\tilde{N}(\td z,\td s).\label{noise_2}
\end{align}
For fixed $h,\epsilon\in(0,1)$ introduce the following random variable
\begin{equation}\label{increment}
\bar{\Delta}_1:=\bar{a}h+\bar{B}\sqrt{h}\xi_I+\sum_{j=1}^{N_h^\epsilon}Y_j^\epsilon,
\end{equation}
and take independent copies $\bar{\Delta}_2,\cdots,\bar{\Delta}_{[T/h]}$, where $\{Y_j^\epsilon\}$ are i.i.d. random variables having distribution $\one_{|z|>\epsilon}\nu(\td z)/\nu(|z|>\epsilon)$, $N_h^\epsilon$ is Poisson with parameter $h\nu(\{|z|>\epsilon\})$, and the coefficients $\bar{a}=a-\int_{|z|>\epsilon}z\nu(\td z),~\bar{B}:=\lb BB^\top+\Sigma_\epsilon\rb^{1/2},~\Sigma_\epsilon=\int_{0<|z|\lq\epsilon}zz^\top\nu(\td z)$. For $t_k=kh,~k=1,\cdots,[T/h]$, write the increments $\Delta_k:=Z_{t_k}-Z_{t_{k-1}}$. Then one may attempt to find a coupling between the standard Euler's approximation
\begin{equation*}
X_{k+1}:=X_k+\sigma(X_k)\Delta_{k+1},~X_0=x_0,
\end{equation*}
and the numerical scheme
\begin{equation}\label{scheme}
\bx_{k+1}:=\bx_k+\sigma(\bx_k)\bar{\Delta}_{k+1},~\bx_0=x_0.
\end{equation}
For that one claims the following statement as an analogue to Lemma 5.2 in \cite{fournier2011SimappLevstodifequ}:
\begin{proposition}\label{coupling}
If $\nu$ satisfies the conditions of Theorem \ref{jump_normal_approx}, then for $\epsilon$ sufficiently small there exist on the same probability space two sequences of i.i.d. random variables $\{\Delta'_k\},~\{\bar{\Delta}'_k\}$, with the same distributions as $\{\Delta_k\}$ and $\{\bar{\Delta}_k\}$ respectively, s.t.
\begin{equation*}
\lb\ex\lv\Delta'_k-\bar{\Delta}'_k\rv^p\rb^\frac{1}{p}\lq C_q\epsilon,
\end{equation*}
for all $k\in\ZZ^+$ and $p\in2\ZZ^+$, and $\ex\Delta'_k=\ex\bar{\Delta}'_k=ah,~\var\Delta'_k=\var\bar{\Delta}'_k=\lb BB^\top+\Sigma_\epsilon\rb h$.
\end{proposition}
\begin{proof}
By Theorem \ref{jump_normal_approx}, for $\epsilon$ sufficiently small there is a standard normal random variable $\xi'_I$ on the same probability space s.t.
\begin{equation*}
\lb\ex\lv\int_0^h\int_{0<|z|\lq\epsilon}z\tilde{N}(\td z,\td s)-h\Sigma_\epsilon^\frac{1}{2}\xi'_I\rv^p\rb^\frac{1}{p}\lq C_q\epsilon,
\end{equation*}
according to the definition of the $\wass_p$ distance. If one sets
\begin{align*}
\Delta'_1:=&ah+BW_h+\int_0^h\int_{0<|z|\lq\epsilon}z\tilde{N}(\td z,\td s)+\int_0^h\int_{|z|>\epsilon}z\tilde{N}(\td z,\td s),\\
\bar{\Delta}'_1:=&ah+BW_h+h\Sigma_\epsilon^\frac{1}{2}\xi'_I+\int_0^h\int_{|z|>\epsilon}z\tilde{N}(\td z,\td s),
\end{align*}
then $\Delta'_1$ has the same law as $\Delta_1$, and $\bar{\Delta}'_1$ has the same law as $\bar{\Delta}_1$. Thus the result follows by taking independent copies.
\end{proof}

Proposition \ref{coupling} can be immediately used to partially recover the main results in \cite{fournier2011SimappLevstodifequ} (Theorem 2.2): the proof is independent of the key coupling result (Lemma 5.2), so one can replace the latter with the proposition above. Hence one can restate those results as follows: 
\begin{theorem}
Suppose $\sigma:\RR^d\to\RR^{d\times q}$ is bounded and Lipschitz, and the L\'evy measure $\nu$ for the L\'evy process \eqref{noise_2} satisfies conditions of Theorem \ref{jump_normal_approx}. Let $\epsilon,h\in(0,1)$, and $\{x_t\}$ be the unique solution to the SDE \eqref{sde_2} for $t\in[0,T]$. Then for $\rho_h(t)=[t/h]h$ and $\epsilon$ sufficiently small, there exists a coupling between $\{x_t\}$ and $\{\bar{X}_{\rho_n(t)}\}$ defined by \eqref{scheme} and \eqref{increment} s.t.
\begin{equation*}
\ex\sup_{t\in[0,T]}\lv x_t-\bar{X}_{\rho_h(t)}\rv^2\lq C_1(h+\epsilon).
\end{equation*}
Moreover, if $\nu(\{|z|>\epsilon\})=0$, i.e. $Z_t=Z_t^\epsilon$ as in \eqref{noise_2}, and $\{\tilde{x}_t^\epsilon\}$ is the unique solution to the continuous SDE $\tilde{x}_t^\epsilon=x_0+\int_0^t\sigma(\tilde{x}_t^\epsilon)\td\tilde{Z}_t^\epsilon$ where $\tilde{Z}_t^\epsilon=at+(BB^\top+\Sigma_\epsilon)^{1/2}W_t$, then there exists a coupling between $x_t$ and $\tilde{x}_t^\epsilon$ s.t.
\begin{equation*}
\ex\sup_{t\in[0,T]}\lv x_t-\tilde{x}_t^\epsilon\rv^2\lq C_2\epsilon.
\end{equation*}
The constants $C_1,C_2$ depend on $d,q,T,|a|,\|B\|,\|\sigma\|_\infty,\Sigma_\epsilon$.
\end{theorem}
Instead of repeating the same arguments of Fournier \cite{fournier2011SimappLevstodifequ}, the reader is referred to the proof of Theorem 2.2 therein. Note that Proposition \ref{coupling} above allows one to replace the $\beta_\epsilon(\nu)$ in Lemma 5.2 with $\epsilon^2$, and the rest of the calculations can be readily generalised to the multi-dimensional case. In particular, under the assumption \ref{alpha-stable} for some $\alpha\in(1,2)$, by choosing $\epsilon=h$ one recovers the mean-square convergence rate $O(h)$ and the computational cost $E_\nu(h)=O(h^{-1}+h^{-\alpha})$ is controlled. The second statement corresponds to Corollary 3.2 in \cite{fournier2011SimappLevstodifequ}. For that, one simply takes $\bar{\Delta}_1=\bar{a}h+\bar{B}\sqrt{h}\xi_I$ instead of \eqref{increment} and $h=\epsilon$, and runs the same argument as in Proposition \ref{coupling}, omitting the big-jump part.

The general case where $\sigma$ is locally Lipschitz with linear growth and only $\int_{\RR^q\setminus\{0\}}1\wedge|z|^2\nu(\td z)<\infty$ is assumed can be treated by the same localisation argument as in Theorem 7.1 in \cite{fournier2011SimappLevstodifequ}, and the mean-square convergence could be generalised to the strong $L^p$-convergence for $p\in2\ZZ^+$ without much trouble. Nevertheless, it needs to be pointed out that the rate of convergence here is optimal for coupling the small jumps only - it might not be so if one can couple the entire L\'evy increment. For the same reason the results achieved in this article cannot be applied to recover Theorem 3.1 in \cite{fournier2011SimappLevstodifequ}. Finally, I believe the conditions of Theorem \ref{jump_normal_approx} can be relaxed to some extent. E.g., one may take a hint from Proposition A.2 in \cite{AsyGraColSpaHomBolEquSofCouPot} that it possibly suffices for $\nu$ to give a suitable portion of mass to the biggest annulus $\Omega_{r_0}$.

\paragraph{Acknowledgement.}I am grateful for the patient guidance of my Ph.D. supervisor, Prof. A. M. Davie, who suggested these problems and gave many helpful comments. I would also like to thank the referees for drawing my attention to several related works that I was not aware of, particularly \cite{BerBouEdgExpCenLimTheTraDis}, \cite{bonis2016RatCenLimTheDifAppviaSteMet} and \cite{AsyGraColSpaHomBolEquSofCouPot}.

\bibliographystyle{acm}
\bibliography{clt_levy}
\end{document}